\documentclass[reqno,12pt]{amsart}
\def\vint_#1{\mathchoice%
          {\mathop{\kern 0.2em\vrule width 0.6em height 0.69678ex depth -0.58065ex
                  \kern -0.8em \intop}\nolimits_{\kern -0.4em#1}}%
          {\mathop{\kern 0.1em\vrule width 0.5em height 0.69678ex depth -0.60387ex
                  \kern -0.6em \intop}\nolimits_{#1}}%
          {\mathop{\kern 0.1em\vrule width 0.5em height 0.69678ex
              depth -0.60387ex
                  \kern -0.6em \intop}\nolimits_{#1}}%
          {\mathop{\kern 0.1em\vrule width 0.5em height 0.69678ex depth -0.60387ex
                  \kern -0.6em \intop}\nolimits_{#1}}}
\def\vintslides_#1{\mathchoice%
          {\mathop{\kern 0.1em\vrule width 0.5em height 0.697ex depth -0.581ex
                  \kern -0.6em \intop}\nolimits_{\kern -0.4em#1}}%
          {\mathop{\kern 0.1em\vrule width 0.3em height 0.697ex depth -0.604ex
                  \kern -0.4em \intop}\nolimits_{#1}}%
          {\mathop{\kern 0.1em\vrule width 0.3em height 0.697ex de pth -0.604ex
                  \kern -0.4em \intop}\nolimits_{#1}}%
          {\mathop{\kern 0.1em\vrule width 0.3em height 0.697ex depth -0.604ex
                  \kern -0.4em \intop}\nolimits_{#1}}}

\usepackage{a4wide}
\usepackage{mathtools}
\usepackage{amsthm}
\usepackage{amssymb}
\usepackage{dsfont}
\usepackage{hyperref}
\usepackage{color}
\usepackage{csquotes}
\usepackage{graphicx}
\usepackage{epsfig,color}
\usepackage{enumitem}
\usepackage[capitalise]{cleveref}
\usepackage{glossaries}
\expandafter\def\csname ver@etex.sty\endcsname{3000/12/31}

\usepackage{autonum}
\usepackage[foot]{amsaddr}
\usepackage{caption}
\usepackage{subcaption}
\usepackage{tikz}
\usetikzlibrary{decorations.pathreplacing,positioning}

\numberwithin{equation}{section}
\newtheorem{theorem}{Theorem}[section]
\newtheorem{lemma}[theorem]{Lemma}
\newtheorem{corollary}[theorem]{Corollary}
\theoremstyle{definition}
\newtheorem{definition}[theorem]{Definition}
\theoremstyle{plain}
\newtheorem{proposition}[theorem]{Proposition}
\theoremstyle{remark}
\newtheorem{rem}[theorem]{Remark}
\newtheorem{example}[theorem]{Example}

\renewcommand{\d}{{\mathrm{\,d}}}
\newcommand{\R}{\mathbb{R}}

\newcommand{\N}{\mathbb{N}}
\newcommand{\Z}{\mathbb{Z}}

\newcommand{\B}{\mathcal{B}}

\newcommand{\opt}{\mathrm{Opt}}
\newcommand{\id}{\mathrm{id}}

\newcommand{\M}{\mathcal{M}}
\renewcommand{\L}{\mathcal{L}}
\newcommand{\Lip}{\mathrm{Lip}}
\newcommand{\var}{\mathrm{Var}}
\newcommand{\essvar}{\mathrm{ess\, Var}}

\makeatletter
\def\namedlabel#1#2{\begingroup
    #2%
    \def\@currentlabel{#2}%
    \phantomsection\label{#1}\endgroup
}
\makeatother

\begin{document}

\title[Absolutely continuous and BV-curves in 1-Wasserstein spaces]
{Absolutely continuous and BV-curves in \\1-Wasserstein spaces}

\author{Ehsan Abedi$^\dagger$}
\author{Zhenhao Li$^\dagger$}
\author{Timo Schultz$^\dagger$$^\ddag$}

\address{$^\dagger$Faculty of Mathematics \\
         Bielefeld University \\
         Postfach 10 01 31 \\
         33501 Bielefeld \\
         Germany.}
\address{$^\ddag$Institute for Applied Mathematics \\
        University of Bonn \\
        Endenicher Allee 60 \\
        53115 Bonn \\
        Germany.}
        
\email{ehsan.abedi@math.uni-bielefeld.de}
\email{zhenhao.li@math.uni-bielefeld.de}
\email{timo.m.schultz@jyu.fi}




\thanks{}


\begin{abstract}
    We extend the result of Lisini (Calc Var Partial Differ Equ 28:85–120, 2007) on the superposition principle for absolutely continuous curves in $p$-Wasserstein spaces to the special case of $p=1$.
    In contrast to the case of $p>1$, it is not always possible to have lifts on absolutely continuous curves. Therefore, one needs to relax the notion of a lift by considering curves of bounded variation, or shortly BV-curves, and replace the metric speed by the total variation measure.
    We prove that any BV-curve in a 1-Wasserstein space can be represented by a probability measure on the space of BV-curves which encodes the total variation measure of the Wasserstein curve.
    In particular, when the curve is absolutely continuous, the result gives a lift concentrated on BV-curves which also characterizes the metric speed.
    The main theorem is then applied for the characterization of geodesics and the study of the continuity equation in a discrete setting. 
\end{abstract}

\maketitle

\section{Introduction}\label{sec:introduction}
Let $(X,d)$ be a complete and separable metric space and $P_p(X)$ be the associated Wasserstein space of order $p\geq 1$, i.e., the space of Borel probability measures on $X$ with finite moment of order $p$, endowed with the (Kantorovitch--Rubinstein--)Wasserstein distance $W_p$.
In \cite{Lisini}, Lisini proved that, for $p>1$, any absolutely continuous curve $(\mu_t) \in \mathcal{AC}^p(I:P_p(X))$ over a compact time interval $I \subset \R$ with finite $p$-energy can be represented by a Borel probability measure $\pi$ on continuous curves $(\gamma_t)$ in $X$, that is, $\pi \in P(C(I:X))$, which satisfies the following properties:
\begin{enumerate}[label=(\roman*), font=\normalfont]
    \item \label{itm:Lisini_thm5_i}
    $\pi$ is concentrated on $\mathcal{AC}^p(I:X) \subset C(I:X)$; 
    \item \label{itm:Lisini_thm5_ii}
    $(e_t)_\#\pi=\mu_t$ for all $t\in I$ (where $e_t$ is the evaluation map, defined by $e_t(\gamma) \coloneqq \gamma_t$); 
    \item \label{itm:Lisini_thm5_iii} 
    the metric derivative $|\dot\mu_t|$ satisfies the following for $\mathcal{L}^1$-a.e. $t\in I$:
    \begin{equation}\label{eq:dotmu_p}
    |\dot\mu_t|^p =  \int |\dot\gamma_t|^p \d\pi(\gamma). 
    \end{equation}
\end{enumerate}
Measures on a path space, like $\pi$ above, are sometimes called \emph{path measures}. \cref{itm:Lisini_thm5_i} tells us that to characterize $(\mu_t)$, we can restrict our attention to a specific set of continuous curves, namely, absolutely continuous curves, or shortly \emph{AC-curves}. 
\cref{itm:Lisini_thm5_ii} ensures that $\pi$ has the desired time-marginals, or in other words, is a \emph{lift} of $(\mu_t)$ to the path space. 
This is also known as a \emph{superposition principle} since the curve of measures $(\mu_t)$ is obtained by superposing individual curves $(\gamma_t)$ in the underlying space.
Finally, \cref{itm:Lisini_thm5_iii} states that the metric speed $|\dot\mu_t|$ in the $p$-Wasserstein space can be obtained by taking the average over the metric speed of the characterizing curves in the base space according to the measure $\pi$. 
Equation \eqref{eq:dotmu_p} can be regarded as a minimality property for $\pi$.
Indeed, for general lifts satisfying \ref{itm:Lisini_thm5_i}-\ref{itm:Lisini_thm5_ii}, one can expect only an inequality $(\leq)$ in \eqref{eq:dotmu_p} (see \cite[Theorem 4]{Lisini}). 
The minimal choice, which achieves equality, is in fact constructed using techniques of optimal transport.
For Wasserstein geodesics, such a lift, often called \emph{optimal dynamical plans}, is constructed in an earlier work by Lott and Villani \cite[Proposition 2.10 and E.6]{LottVillani2009} for the case of $p=2$ and in complete locally compact length spaces.
In Lisini's work \cite[Theorem 5]{Lisini}, which local compactness is no longer required, the lift is constructed for general absolutely continuous curves in $p$-Wasserstein spaces, $p>1$, and in particular is used for the characterization of the geodesics.
Later in \cite{Lisini2016}, Lisini also extends the results above to the so-called Wasserstein--Orlicz distance, where the usual cost function $d^p$ is replaced by a more general function $\psi$ with suitable properties.
This extension, however, does not cover the case $d^1$.

In this paper, we study the peculiar case of $p=1$, where the cost function $d^p$ in the definition of the Wasserstein distance loses its strict convexity. 
We first provide a simple example (see \cref{ex:theBeginning} below) in which an absolutely continuous curve in an $1$-Wasserstein space cannot be lifted to a measure $\pi$ on continuous curves.
Nonetheless, we show that a similar superposition result still holds if we relax the notion of lifts.
More precisely, we need to consider a larger class of curves, namely, curves of bounded variation, or shortly \emph{BV-curves} (see also \cref{exp:AC_not_enough}).

When considering the case $p>1$, it is well known that the space of absolutely continuous curves with finite $p$-energy is closely connected to Sobolev space of order 1 with finite $p$-norm via the following ``identification-inclusion'' relationship 
\begin{align}\label{eq:inclusion_AC}
    W^{1,p}(I:X) \simeq \mathcal{AC}^p(I:X) \subset C(I:X),
\end{align}
which succinctly indicates that every Sobolev curve can be identified with an absolutely continuous representative.
Additionally, we have the Borel inclusion
of absolutely continuous curves into the space of continuous curves equipped with the topology of
uniform convergence, which turns it into a Polish space. In the present paper, where we study the case $p=1$, these are replaced by the following
\begin{align}\label{eq:inclusion_BV}
    BV(I:X) \simeq \mathcal{BV}(I:X) \subset D (I:X).
\end{align}
Here $BV(I:X) $ denotes the space of all BV-curves. As an analogue to \eqref{eq:inclusion_AC}, every BV-curve can be identified through a Borel selection map with a Càdlàg (right-continuous and left-limited) curve of bounded variation.
The space of such curves is denoted by $\mathcal{BV}(I:X)$, which is a Borel subset of the larger space of all possible Càdlàg curves denoted by $D (I:X)$. 
The space $D (I:X)$ can be equipped with a specific metric, which turns it into a Polish space, known as \emph{Skorokhod space}.
It is worth mentioning that in restriction to $C (I:X)$, the Skorokhod topology is exactly the topology of
uniform convergence. 
In short, we view BV-curves as a Borel subset, up to choosing a representative, of Skorokhod space.

Even though the metric derivative of BV-curves $u\in BV(I:X)$ exists almost everywhere, as does so for AC-curves, it does not completely capture their ``speed.'' The natural replacement for metric derivative in this situation is the so-called \emph{total variation measure} $|Du| \in \mathcal{M}(I)$, which takes also the singular part of the speed, in particular jumps of the curves, into account.
Here $\mathcal{M}(I)$ is the set of all positive measures over $I$ and we will use $\L^n$ to denote $n$-dimensional Lebesgue measure. 

\vspace{5pt}

\noindent \textbf{Main result.} 
In \cref{thm:lift_BV}, we prove that any $(\mu_t)\in \mathcal{BV}(I:P_1(X))$ can be represented by a Borel probability measure $\tilde{\pi}$ on Càdlàg curves in $X$, that is,  $\tilde{\pi}\in P(D(I:X))$, which satisfies the following properties:
\begin{enumerate}[label=(\roman*), font=\normalfont]
    \item $\tilde{\pi}$ is concentrated on $\mathcal{BV}(I:X)\subset D(I:X)$;
    \item $(e_t)_\#\tilde{\pi}=\mu_t$ for all $t\in I$;
    \item the total variation measure $|D\mu|\in \M(I)$ of $(\mu_t)$ satisfies
    \begin{equation}\label{eq:Dmu_intro}
         |D\mu|=\int|D\gamma|\d\tilde{\pi}(\gamma).
    \end{equation}
\end{enumerate}
Moreover, the absolutely continuous part $|\dot\mu|\mathcal{L}^1$ in the Lebesgue(--Radon--Nikodym) decomposition of the measure $|D\mu|$, given by the metric derivative (see discussion in \cref{subsec:BV-curves}), satisfies
\begin{align}\label{eq:mudot_intro}
     |\dot\mu_t|   = \lim_{h \rightarrow 0 } \int \frac{  d (\gamma_{t}, \gamma_{t+h})}{|h|} \d\tilde{\pi}(\gamma)
\end{align}
for $\mathcal{L}^1$-a.e. $t\in I$. 
In particular, if $(\mu_t) \in \mathcal{AC}^1(I:P_1(X))$, then $|D\mu|=|\dot\mu|\mathcal{L}^1$ and \eqref{eq:mudot_intro} characterizes the metric speed $|\dot\mu_t|$. 
\\
Equation \eqref{eq:Dmu_intro} is interpreted as equality of measures, i.e., for any (non-negative) Borel function $f\colon I \to \mathbb{R}$, we have $\int_I f(t)\d|D\mu|(t)=\int\int_I f(t)\d|D\gamma|(t)\d\tilde{\pi}(\gamma)$.
\cref{thm:BVSk}, which we prove first, indicates that \eqref{eq:Dmu_intro} can be viewed as an optimality condition among all lifts of $(\mu_t)$, as in the case of $p>1$.
To construct $\tilde{\pi}$, we use optimal mass transport as in \cite{Lisini} with modifications for BV-curves.

\vspace{5pt}

\noindent \textbf{A motivating example.} Here we present an elementary example of an absolutely continuous curve in the $1$-Wasserstein space over $\R$, for which it is impossible to have lifts on continuous curves (also discussed briefly in \cite[Remark 3.2]{Lisini2016}). Still, we construct a lift on discontinuous BV-curves. This provides a first insight into our results.
\begin{example}\label{ex:theBeginning}
Consider a curve of probability measures on $\R$ defined as 
\begin{align}\label{eq:mu_t_motivate_example}
    \mu_t \coloneqq (1-t)\delta_0+t\delta_1, \quad t \in I = [0,1]. 
\end{align}
This is a basic situation where the mass is ``teleported'' from 0 to 1, but not continuously ``transported,'' as shown in \cref{fig:motivate_example} (left). 
First of all, observe that for any $t,s\in I$,
\begin{equation}
    W_p^p (\mu_t, \mu_s)  = |t-s|  W_p^p (\mu_1, \mu_0) = |t-s|, 
\end{equation}
and thus, the metric derivative in $p$-Wasserstein space,
$$ |\dot\mu_t|= \lim_{h \to 0 } \frac{W_p(\mu_{t+h}, \mu_t)}{|h|} = \lim_{h \to 0 } \frac{|h|^{1/p}}{|h|},$$
only exits for $p=1$.
Therefore, $(\mu_t) \notin \mathcal{AC}^p(I:P_p(\R)) $ for all $p\in (1, \infty )$. Nevertheless $(\mu_t) \in \mathcal{AC}^1(I:P_1(\R)) $ and it is even a constant-speed geodesic in 1-Wasserstein space between $\delta_0$ and $\delta_1$.
\\
It is clear that there is no measure $\pi$ on the set of continuous curves, i.e.,  a measure in $P(C(I:\R))$, such that $\mu_t = (e_t)_\# \pi $ for all $t \in I$.
However, we do claim that there exists a measure $\pi \in P(D(I:\R))$ concentrated on the set of $\mathcal{BV}$-curves  such that $\mu_t = (e_t)_\# \pi $ for all $t \in I$ and moreover it enjoys the optimally property
\begin{equation}\label{eq:dotmu_example}
    \int_{a}^{b} |\dot\mu_t| \d t = \int
    |D\gamma |([a,b]) \d \pi (\gamma )
\end{equation}
for any $a,b \in I$ with $a<b$.
Comparing with \eqref{eq:Dmu_intro}, we point out that the left-hand side of the equation above is nothing but $|D\mu|([a,b])$ since $(\mu_t)$ in this simple example is absolutely continuous.
\\
To construct $\pi$, let us label particles standing at position $x=0$ at time 0 with a real-valued parameter denoted by $\alpha\in [0,1]$. 
These particles gradually jump to position $x=1$ and since the rate of mass discharge is constant, we would expect that jumps happen uniformly in time.
Let the particle with label $\alpha$ jump from 0 to 1 at time $\alpha$.
Then its path is simply expressed using the indicator function as follows 
\begin{equation}\label{eq:gamma_example1}
    t \mapsto \gamma_t^{(\alpha)} \coloneqq \mathds{1}_{[\alpha, 1]}(t).
\end{equation}
Some sample paths are plotted in \cref{fig:motivate_example} (right).
Now, we consider a uniform measure over $\alpha$ (as jumps happen uniformly) and consequently construct a path measure $\pi$ by 
\begin{equation}\label{eq:pi_example1}
    \pi \coloneqq (\gamma^{(\cdot)})_\# \mathcal{L}^1 |_{[0,1]}.
\end{equation}
We show now that $\pi$ has the desired time-marginals and satisfies \eqref{eq:dotmu_example} as well.
As for the first claim, notice that for any Borel subset $B \in \mathcal{B}(\R)$, we can write
\begin{align}
    (e_t)_\# \pi (B)
    & = \int \mathds{1}_{B} (e_t(\gamma)) \d\pi(\gamma) 
      = \int_{0}^1 \mathds{1}_{B} (\gamma_t^{(\alpha)}) \d \alpha \\
    & = \int_{0}^1 \mathds{1}_{B} (\mathds{1}_{[\alpha, 1]}(t)) \d \alpha
      = \int_{0}^1 \mathds{1}_{B} (\mathds{1}_{[0, t]}(\alpha)) \d \alpha \\
    & = (1-t) \delta_0 (B) + t \delta_1(B) = \mu_t(B)
\end{align}
where we first used the definition of push-forward and then substituted \eqref{eq:pi_example1} and \eqref{eq:gamma_example1}.
As for the second claim \eqref{eq:dotmu_example}, we start from the right-hand side and write
\begin{align}
    \int |D\gamma |([a,b]) \d \pi (\gamma )
    & = \int_{0}^{1} |D\gamma^{(\alpha)} |([a,b]) \d \alpha \\
    & = \int_{0}^{1} \delta_\alpha ([a,b]) \d \alpha = |b-a| = \int_{a}^{b} |\dot\mu_t| \d t, 
\end{align}
which proves the claim.
\\
As already mentioned, $(\mu_t)$ here is a constant-speed geodesic connecting $\delta_0$ to $\delta_1$.
In fact, there are infinitely many constant-speed $W_1$-geodesics between $\delta_0$ to $\delta_1$. In \cref{exp:periodic_extension_2}, we present a relatively general way of how one can construct different geodesics.  

\begin{figure}[h]
\begin{tikzpicture}
    \draw [lightgray] [line width=4] (0,0) -- (0,2*4/5);
    \draw [lightgray] [line width=4] (3,0) -- (3,2*1/5);
    \draw [black] [line width=0.9] (0,0) -- (0,2*3/5);
    \draw [black] [line width=0.9] (3,0) -- (3,2*2/5);
    \draw (0,2*4/5) node[anchor=east] {\scriptsize \color{gray} $\mu_s$};
    \draw (0,2*3/5) node[anchor=east] {\scriptsize \color{black} $\mu_t$};
    \draw [very thin] [-] (3.3,2*1/5) -- (3.3,2*2/5);
    \draw [very thin] (3.3-0.08,2*2/5) -- (3.3+0.08,2*2/5);
    \draw [very thin] (3.3-0.08,2*1/5) -- (3.3+0.08,2*1/5);
    \draw (3.3,2*1.5/5) node[anchor=west] {\scriptsize $t-s$};
    \draw (-0.5,0) -- (4.1,0);
    \filldraw [black] (0,0) circle (1.4pt);
    \draw (0,-0.1) node[anchor=north] {\scriptsize $x=0$};
    \filldraw [black] (3,0) circle (1.4pt);
    \draw (3,-0.1) node[anchor=north] {\scriptsize $x=1$};
    \draw [gray][dashed][->] (0.1,2) .. controls (1,2.7) and (2,2.2) .. (2.9,1.1);
    \draw [<->] (6,2.5) node [left] {\scriptsize $x$} -- (6,0) -- (9.7,0) node [below] {\scriptsize time};
    \draw (6,-0.1) node[anchor=north] {\scriptsize $0$};
    \draw (6,-0.08) rectangle (6,0);
    \draw (6.6,-0.14) node[anchor=north] {\scriptsize $s$};
    \draw (6.6,-0.08) rectangle (6.6,0.08);
    \draw (7.2,-0.1) node[anchor=north] {\scriptsize $t$};
    \draw (7.2,-0.08) rectangle (7.2,0.08);
    \draw (9,-0.1) node[anchor=north] {\scriptsize $1$};
    \draw (9,-0.08) rectangle (9,0.08);
    \draw (6,0.1) node[anchor=east] {\scriptsize $0$};
    \draw (6-0.08,0) rectangle (6,0);
    \draw (6,2) node[anchor=east] {\scriptsize $1$};
    \draw (6-0.08,2) rectangle (6+0.08,2);
    \draw [dashed][lightgray] (6.6,0) -- (6.6,2);
    \draw [dashed][lightgray] (6.7,0) -- (6.7,2);
    \draw [dashed][lightgray] (6.8,0) -- (6.8,2);
    \draw [dashed][lightgray] (7.0,0) -- (7.0,2);
    \draw [dashed][lightgray] (7.1,0) -- (7.1,2);
    \draw [dashed][lightgray] (7.2,0) -- (7.2,2);
    \draw [dashed][black] (6.9,0) -- (6.9,2);
    \draw [lightgray] (6.6,0) circle [radius=0.04];
    \draw [lightgray] (6.7,0) circle [radius=0.04];
    \draw [lightgray] (6.8,0) circle [radius=0.04];
    \draw [lightgray] (7.0,0) circle [radius=0.04];
    \draw [lightgray] (7.1,0) circle [radius=0.04];
    \draw [lightgray] (7.2,0) circle [radius=0.04];
    \draw [fill=lightgray,lightgray] (6.6,2) circle [radius=0.04];
    \draw [fill=lightgray,lightgray] (6.7,2) circle [radius=0.04];
    \draw [fill=lightgray,lightgray] (6.8,2) circle [radius=0.04];
    \draw [fill=lightgray,lightgray] (7.0,2) circle [radius=0.04];
    \draw [fill=lightgray,lightgray] (7.1,2) circle [radius=0.04];
    \draw [fill=lightgray,lightgray] (7.2,2) circle [radius=0.04];
    \draw [lightgray][thick] (6,0) -- (6.6,0);
    \draw [lightgray][thick] (6,0) -- (6.7,0);
    \draw [lightgray][thick] (6,0) -- (6.8,0);
    \draw [lightgray][thick] (6,0) -- (7.0,0);
    \draw [lightgray][thick] (6,0) -- (7.1,0);
    \draw [lightgray][thick] (6,0) -- (7.2,0);
    \draw [lightgray][thick] (6.6,2) -- (9,2);
    \draw [lightgray][thick] (6.7,2) -- (9,2);
    \draw [lightgray][thick] (6.8,2) -- (9,2);
    \draw [lightgray][thick] (7.0,2) -- (9,2);
    \draw [lightgray][thick] (7.1,2) -- (9,2);
    \draw [lightgray][thick] (7.2,2) -- (9,2);
    \draw [fill=black] (6.9,2) circle [radius=0.04];
    \draw [black][thick] (6,0) -- (6.9,0);
    \draw [black][thick] (6.9,2) -- (9,2);
    \draw [fill=white] (6.9,0) circle [radius=0.04];
\end{tikzpicture}
\captionsetup{font=scriptsize}
\caption{(\cref{ex:theBeginning}) \textbf{Left:} measure \eqref{eq:mu_t_motivate_example} at times $s=1/5, \, t=2/5$. \textbf{Right:} sample curves $\gamma^{(\alpha)}$ in the lift $\pi$ whose jumps occur at time $\alpha \in [s,t]$. The curve with $\alpha = (s+t)/2$ is highlighted.}
\label{fig:motivate_example}
\end{figure}
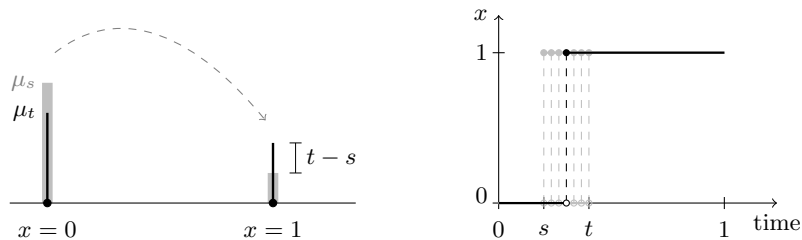
\end{example}

\noindent
\textbf{Applications.}  As a direct application of Theorems \ref{thm:BVSk} and \ref{thm:lift_BV}, we characterize BV-curves in 1-Wasserstein spaces. 
Furthermore, we characterize what we call BV-geodesics, i.e., variation minimizing curves, in 1-Wasserstein spaces. 
With the understanding (thanks to \cref{thm:lift_BV}) of continuity and metric derivatives of Wasserstein curves of bounded variation, we then distinguish continuous length minimizing and constant speed geodesics among all BV-geodesics. 
We also discuss why the characterizing absolutely continuous curves in 1-Wasserstein spaces using their lifts still remains challenging.
    
As seen in \cref{ex:theBeginning}, superposing discontinuous curves might result in continuity in the 1-Wasserstein space. 
On the other hand, continuous curves will always lead to a continuous Wasserstein curve. 
We investigate the relation between the regularity of the curves at the level of the base space and at the level of the Wasserstein space. 
The observations can be summarized as follows: superposing curves can only increase the regularity or, to put it differently, irregularities may average out when superposing, see \cref{table:gamma_mu}.
    
Finally, we study the continuity equation in a discrete setting.
More precisely, using the lift coming from \cref{thm:lift_BV}, we show that for any absolutely continuous curve $(\mu_t)$ living in a bounded subset of a (topologically) discrete metric space, there exists $(v_t)$, a suitable discrete analogue of a time-dependent vector field, such that $(\mu_t,v_t)$ satisfy the discrete continuity (or current) equation.
We conclude with a discussion on a discrete Benamou--Brenier formula and on challenges arising if one is interested not only in the metric structure of the discrete space but also in an additional graph structure.

\vspace{5pt}

\noindent \textbf{Organisation of the paper.} The remainder of the paper is structured as follows:
\begin{itemize}
    \item In \cref{sec:preliminaries}, we provide preliminary concepts concerning BV-curves in metric spaces and Skhorokhod space.
    Additionally, we prove equivalent definitions of BV-curves in \cref{thm:DefBVcurves}.
    Such a result (which we did not easily find it in the literature) makes it more convenient to work with BV-curves in different situations. 
    \item In \cref{sec:results}, we present and prove the main results, \cref{thm:BVSk} and \ref{thm:lift_BV}. We then provide some examples to shed light on the main results.
    \item In \cref{sec:applications}, we apply the main theorems to characterize BV-curves and geodesics, understand better the regularity of curves in superposition, and finally study the continuity equation in a discrete setting. 
\end{itemize}

\vspace{5pt}

\noindent \textbf{Acknowledgement.}
We thank Matthias Erbar for helpful suggestions and stimulating discussions on related topics. We also thank Tapio Rajala for valuable comments on the paper.
The third named author acknowledges the support provided by the Deutsche Forschungsgemeinschaft (DFG, German Research Foundation) - through SPP 2026 Geometry at Infinity.
The authors also thank the anonymous referee for providing detailed feedback on the manuscript.

\vspace{5pt}

\noindent \textbf{Data availability.} Data sharing is not applicable to this article as no datasets were generated or analyzed during the current study.

\section{Preliminaries}\label{sec:preliminaries}
\subsection{Summary of main notation}\label{subsec:notation}

 Throughout the paper, we consider a complete and separable metric space $(X,d)$ and a compact time interval $I \subset \mathbb{R}$. Without loss of generality, we fix $I=[0,1]$. 
Two main path spaces are 
the space of continuous paths $C(I:X)$ equipped with the topology of uniform convergence, and the space of Càdlàg paths $D(I:X)$ equipped with the Skorokhod topology. For $p\geq 1$, $W^{1,p}(I:X)$ denotes the Sobolev space and $\mathcal{AC}^p(I:X)$ denotes the set absolutely continuous curves with finite $p$-energy (recall the relationship \eqref{eq:inclusion_AC} for $p>1$). $BV(I:X)$ is the set of curves of bounded variation, and $\mathcal{BV}(I:X)$ is the set of Càdlàg curves of bounded variation (recall the relationship \eqref{eq:inclusion_BV} and see \cref{def:BV,def:cadlag}).
$\mathcal{M}(I)$ is the set of all positive measures over $I$ and $\mathcal{L}^n$ is the $n$-dimensional Lebesgue measure.
        
For $t\in I$, $u\mapsto e_t (u) \coloneqq u_t$ is the evaluation map, where $u\colon I\to X$ is a path. 
The variation measure of a BV-curve $u$ is denoted by $|Du|$ and the density of the absolutely continuous part at point $t$ with respect to the Lebesgue measure is denoted  by $|\dot{u}| (t)$. If it exists, the metric derivative of  a curve $u$ at time $t$ is denoted by $|\dot{u}_t|$ (see \cref{def:varmeasure,eq:metric_derivative}). In fact, $|\dot{u}| (t) = |\dot{u}_t|$, as stated in \cref{lemma:metric_derivative}.
        
The $\sigma$-algebra of Borel sets of $X$ is denoted by $\mathcal{B}(X)$. We denote by $P(X)$ the space of Borel probability measures on $X$, and by $P_p(X) \subset P(X)$, $p\geq 1$,  its subset of measures with finite $p$-th moment. The space $P_p(X)$ is endowed with (Kantorovitch--Rubinstein--)Wasserstein metric $W_p$. Given a map $T: X \to Y$ between two measurable spaces and a probability measure $\mu \in P(X)$, the push-forward measure (or the image measure) is denoted by $T_{\#} \mu \in P(Y)$.

\subsection{BV-curves in metric spaces}\label{subsec:BV-curves}
In this subsection, we recall basic definitions and notions related to curves of bounded variation.  
$L^1(I:X)$ denotes the space of all maps $u \colon I \to X$ such that $\int_I d(u (t), \bar{x}) \d t < \infty$ for some (and thus every) $\bar{x} \in X$.

\begin{definition}[Variation]
  The \emph{pointwise variation} of a function $u : I\rightarrow X$ on any subset $J \subset I$ is defined as
  $$ \var^{} (u ; J)\coloneqq\sup\left\{\sum_{i=0}^k d(u(t_i), u
     (t_{i + 1})) \big| t_0 <\cdots < t_{k + 1_{}}, \{ t_j\}_{ 0 \leq j \leq k + 1}  \subset J \right\}, $$
  and its \emph{essential variation} is defined as
  $$\essvar (u ; J) \coloneqq \inf \left\{ \var (v ; J) |u = v \text{ a.e. on } J \right\}.$$
\end{definition}

\begin{definition}[BV-curves]\label{def:BV}
   We call $u \in L^1(I:X)$ a curve of bounded variation, or shortly a BV-curve, if $$\essvar (u) \coloneqq\essvar(u ; I)< \infty.$$
    We use $BV(I:X)$ to denote the space of all BV-curves.
\end{definition}
For a non-decreasing function $f \colon I \to \R$, we can define its \emph{variation measure} as the Lebesgue--Stieltjes measure $|Df|$ given by (see \cite[Section 6.3.3]{stein2009real})
$$|Df| ((a, b)) = f (b^-) - f (a^+).$$
Using this, we can generalize the notion of variation measure in general metric spaces: 
\begin{definition}[Variation measure]\label{def:varmeasure}
  Let $u\in BV(I:X)$. The variation measure of $u$ is defined as the Lebesgue--Stieltjes measure $|Du|$ induced by the non-decreasing function $V \colon I \to \R$ defined as $V(t) \coloneqq \essvar (u ; (0, t))$.
\end{definition}
By Lebesgue(--Radon--Nikodym) decomposition, the variation measure can be written as
\begin{equation}\label{eq:Leb_dec_var_measure}
    |Du| = |\dot u| \mathcal{L}^1+ |Du|^{C} + |Du|^{J},
\end{equation}
where $|\dot u|$ denotes the Radon--Nikodym derivative of the variation measure with respect to the Lebesgue measure, $|Du|^{J}$ is the purely atomic part, or the ``jump part,'' and the remaining term $|Du|^{C}$ is the continuous singular part or the ``Cantor part''. 
The density $|\dot{u}|(t)$ actually coincides almost everywhere with the \emph{metric derivative} $|\dot{u}_t|$, hence the notation (see \cref{lemma:metric_derivative} at the end of this subsection).
In this paper, we preserve the term \emph{metric speed} only for the metric derivative of \emph{absolutely continuous} curves. We however do not claim this to be a common practice in the literature.

\begin{definition}[Càdlàg curves]\label{def:cadlag}
   A curve $u\colon I\to X$ is called a Càdlàg curve if it is right-continuous with left-limits, i.e., for every $t \in I $ we have that
\begin{align}
    u(t)=\lim_{s\searrow t}u(s), \, \exists \lim_{s\nearrow t}u(s).
\end{align}
We use $D(I:X)$ to denote the set of all Càdlàg curves and $\mathcal{BV}(I:X)$ to denote Càdlàg curves of bounded variation. 
\end{definition}

The reason for introducing Càdlàg curves here is that any BV-curve admits a Càdlàg representative, i.e., they coincide $\mathcal{L}^1$-a.e. This is stated in the lemma below:

\begin{lemma}[Càdlàg-representation of BV-curves]\label{lemma:normalizedBV}
\indent\begin{enumerate}
    \item\label{item:2.5.1} Let $u\in D(I:X)$. For any $0\leq a \leq b\leq 1$, 
  \begin{equation}\label{eq:PBVachieved}
    \var (u ;(a,b))= \essvar(u;(a,b))=|Du|((a,b))
  \end{equation}
  and if $u$ is left-continuous at $t=1$, $\var(u;I)=\essvar(u;I)$.
  \item Any $u\in BV(I:X)$ admits a representative $\tilde{u}\in D(I : X)$ such that
  \begin{equation}
  \essvar (u) = \var (\tilde{u}).
  \end{equation}
\end{enumerate} 
\end{lemma}

\begin{proof}
It suffices to consider the nontrivial case when $u$ has finite variation. Since Càdlàg representation if exists must be unique (up to the value at $t=1$), we only need to show that any BV-curve has a Càdlàg representation satisfying \eqref{eq:PBVachieved}.\\
Let $u\in BV(I:X)$. By definition, there is a sequence of $u_n:I\rightarrow X$, such that $u_n=u$ a.e. and $\essvar(u)\swarrow\var(u_n)<\infty$. For each $n$, the function $t \mapsto \var (u_n ; (0, t))$ is non-decreasing so it has left and right limits at each $t \in (0, 1)$ and is continuous at all $t \in I \setminus N_n$, where $N_n$ is at most countable. Then by the completeness of $(X, d)$, $u_n$ must have left and right limits at each $t\in (0, 1)$ and be continuous at all $t \in I \setminus N_n$ as well. As the family $\{u_n\}$ coincides a.e., at each $t \in I$, the right limit $u_n (t^+)$ equals for all $n$, which will be denoted by $\tilde{u} (t)$. Clearly, on $I \setminus \cup_n N_n$, $u_n=\tilde{u}$ for all $n$, ensuring $\tilde{u}$ is a representative of $u$.\\
Notice that if a function has right limits everywhere then its corresponding right limit function is right continuous. Therefore, $\tilde{u}$ is right-continuous on $[0, 1)$.
And the existence of left limits is a direct consequence of $\tilde{u}\in BV$.
  So for its variation, given any $\varepsilon > 0$ and partition $0 \leq t_0 < \cdots < t_{k + 1} \leq 1$, we can show for any $n \in \mathbb{N}$,
  \begin{align}
    \sum_{i = 0}^k d (\tilde{u} (t_i), \tilde{u} (t_{i + 1})) & =  \sum_{i =0}^k d (u_n (t_i^+), u_n (t_{i + 1}^+))\\
    & \leq  \sum_{i = 0}^k d (u_n (t_i^{} + r_n), u_n (t_{i + 1}^{}+r_n))+ \frac{\varepsilon}{2^i}\\
    & \leq  \var (u_n) + 2\varepsilon , \label{eq:finitevar}
  \end{align}
  where $0<r_n \ll 1 - t_k$ and assume $u_n (t) \coloneqq u_n(1^-)$ whenever $t\geq 1$ (in fact, without loss of generality, each $u_n$ can be chosen as left-continuous at $t= 1$ as this never increases $\var (u_n)$). After taking supremum of division and passing $\varepsilon$ to $0$, one concludes
  \[ 
  \var(\tilde{u})\leq\liminf_{n \rightarrow \infty}\var(u_n)=\essvar(u).
  \]
 The case for general sub-interval $(a,b)\subset I$ follows by the same argument. Finally, from the definition of variation measure, 
\begin{align}
     |Du|((a,b))&=\lim_{r\searrow 0} \left[ \essvar(u;(0,b-r))-\essvar(u;(0,a+r)) \right]\\
     &=\lim_{r\searrow 0}\left[\var(u;(0,b-r))-\var(u;(0,a+r)) \right]\\
  & = \lim_{r\searrow 0}\var (\tilde{u};[a+r,b - r))=  \var (\tilde{u};(a,b))
\end{align}
where the last line comes directly from the definition of pointwise variation.
\end{proof}

\begin{rem}\label{remark:extendBV}
Given $E\subset I$ with $\L^1(E)=1$, if a function $u$ defined on $E$ has finite pointwise variation, then we can extend $u$ to $I $ with $\var(u;E)=\var(u;I)$.\\
Indeed, arguing as in \cref{lemma:normalizedBV}, at each $t\in I$
\begin{equation}
    \lim_{E\ni \tau\searrow t}u(\tau)
\end{equation}
exists and for $t\in I\setminus E$ we define $u(t)$ as the above limit. On $I\setminus E$, $u$ is right-continuous so the variation will not increase after extension.
\end{rem}
\begin{lemma}\label{lemma:lscPBV}
The function $\essvar\colon L^1(I:X)\rightarrow [0,+\infty]$ is lower semi-continuous.
\end{lemma}
\begin{proof}
Let $(u_n)_{n\in\N}\subset L^1(I:X)$ be a sequence converging to $u\in L^1(I:X)$, that is
\begin{equation}
\int_{I}d(u_n(t),u(t))\d t\rightarrow 0,\quad n\rightarrow \infty.
\end{equation}
Without loss of generality, we may assume that $\essvar (u_n)<\infty$ for all $n\in \N$.
By \cref{lemma:normalizedBV}, we can assume each $u_n$ to achieve \eqref{eq:PBVachieved}. As $(X,d)$ is complete, up to picking a subsequence, $u_n(t)$ converges to $u(t)$ on some $E\subset I$ with full measure, yielding
\begin{equation}
\var(u;E)\leq\liminf_{n\rightarrow \infty}\var(u_n)=\liminf_{n\rightarrow \infty}\essvar(u_n).
\end{equation}
By \cref{remark:extendBV}, $\essvar(u)$ is bounded by $\var(u;E)$ and hence $\essvar$ is lower semi-continuous.
\end{proof}

To end this subsection, we briefly comment on the relation between the variation measure and the metric derivative. 
Recall that the \emph{metric derivative} of $u:I\to X$ at time $t$ is defined by 
\begin{equation}\label{eq:metric_derivative}
    |\dot{u}_t| \coloneqq \lim_{h\to 0}\frac{d(u(t),u(t+h))}{|h|}
\end{equation}
whenever the above limit exists.
\begin{lemma}\label{lemma:metric_derivative}
    Let $u\in\mathcal{BV}(I:X)$. Then for $\L^1$-a.e. $t \in I$, the metric derivative $|\dot{u}_t|$ exists and
    \begin{equation}
       |\dot{u}_t|=|\dot{u}|(t)=\lim_{h\to 0}\frac{|Du|([t,t+h])}{|h|},
    \end{equation}
    where $|\dot{u}|$ is the density in the decomposition \eqref{eq:Leb_dec_var_measure}.
\end{lemma}
\begin{proof}
It is known, e.g. from \cite[Theorem 2.2]{Ambrosio1990MetricSV}, that the metric derivative $|\dot{u}_t|$ exists almost everywhere and equals to the density $|\dot{u}|(t)$.
The second equality is a general fact for measures on $\R$ by the following argument.  
Assume that $\mu$ is a locally finite measure on $\R$ with the decomposition $\mu=\rho\L^1+\mu^s$, where $\mu^s\perp \L^1$. 
By the Lebesgue differentiation theorem, it suffices to show 
\begin{equation}
     \lim_{h\to 0}\frac{\mu^s([t,t+h])}{|h|}=0\quad \textrm{for } \L^1\text{-a.e. }t \in I.
\end{equation}
If else, there exists a Borel set $T\subset \R$ and some $c>0$ such that $\L^1(T)>0$ and
\begin{equation}
    \limsup_{h\to 0}\frac{\mu^s([t,t+h])}{|h|}>c,\quad  \forall t\in T.
\end{equation}
Based on the standard differentiation theorem of measures (cf. \cite[Theorem 2.4.3]{Ambrosio-Tilli}), $\mu^s(T)>0$, which contradicts the fact that $\mu^s$ and $\L^1$ are mutually singular.
\end{proof}

\subsection{Skorokhod space}\label{sec:Skorokhod}
As alluded in the previous subsection, Càdlàg curves are important in the study of BV-curves. 
The space of Càdlàg curves $D (I:X)$ can be equipped with a metric, known as Skorokhod metric, which turns it into a complete and separable space, known as Skorokhod space.
Recall from \eqref{eq:inclusion_AC}-\eqref{eq:inclusion_BV} that $D (I:X)$ with the Skorokhod topology plays the role of $C (I:X)$ with the topology of uniform convergence.
The goal of this subsection is to recall necessary notions of Skorokhod space.

\begin{definition}[Skorokhod space] For two curves $\gamma,\tilde{\gamma} \in D(I:X)$, define a distance
    \begin{align}
        d_{Sk}(\gamma,\tilde{\gamma}) \coloneqq \inf_\lambda\max\{\lVert\lambda\rVert_B, d_{\sup}(\gamma,\tilde{\gamma}\circ \lambda)\},
    \end{align}
    where the infimum runs over all increasing homeomorphisms $\lambda\colon I\to I$, and
    \begin{align}
        \lVert\lambda\rVert_B\coloneqq \sup_{0\le s<t\le 1} \big\lvert \log\big(\frac{\lambda(t)-\lambda(s)}{t-s}\big)\big\rvert.
    \end{align}
    The set $D(I:X)$ equipped with the distance $d_{Sk}$ is called the Skorokhod space.
\end{definition}

By definition, a sequence of curves $\gamma_n\in D(I:X)$ converges to $\gamma\in D(I:X)$ if and only if there exists functions $\{\lambda_n\}$ such that $d_{\sup}(\gamma,\gamma_n\circ \lambda_n)\to 0$ and $ \lVert\lambda_n\rVert_B\to 0$ where the latter ensures $|\lambda_n(t)-t|\to 0$ uniformly in $t$.
\begin{theorem}[Billingsley-Skorokhod]
    Let $X$ be complete and separable metric space. Then the Skorokhod space $(D(I:X),d_{Sk})$ is complete and separable.
\end{theorem}
The proof can be found in \cite[Section 12]{Billingsley}, where the author only studied $D(I:\R)$ but the argument holds exactly the same replacing Euclidean space with general complete and separable metric space.

The next lemma shows that the topology induced by the distance $d_{Sk}$ is finer than $L^1$-topology.
\begin{lemma}\label{lemma:topoL1-Sk}
Let $(\gamma_i)_i$ and $\gamma$ be in $D(I:X)$ such that $\gamma_i\to\gamma$ in $\big(D (I:X),d_{Sk}\big)$. Then $\gamma_i\to \gamma$ in $L^1(I:X)$.
\end{lemma}
\begin{proof}
Let $\lambda_i$, $\varepsilon_i\searrow 0$ so that $d_{\sup}(\gamma\circ\lambda_i(t),\gamma_i(t))\le \varepsilon_i$. Then
\begin{align}
    \int_I d(\gamma(t),\gamma_i(t))\d t&\le \int_Id(\gamma_i(t),\gamma\circ\lambda_i(t))\d t+\int_I d(\gamma\circ \lambda_i(t),\gamma(t))\d t
    \\ &\le \varepsilon_i+\int_I d(\gamma\circ\lambda_i(t),\gamma(t))\d t\to 0,
\end{align}
when $i\to\infty$. 
Here the last term converges to zero due to the almost everywhere continuity of Càdlàg curves (cf. \cite[Lemma 12.1]{Billingsley}) and the dominated convergence theorem.
\end{proof}

\begin{lemma}\label{lma:Boreleval}
The Borel $\sigma$-algebra $\B(D(I:X))$ of the Skorokhod space is equal to the $\sigma$-algebra generated by the evaluation maps. More generally, 
\begin{align}
    \B(D(I:X))=\sigma(e_t: t\in T),
\end{align}
where $T\subset I$ is an arbitrary dense subset of $I$ for which $1\in T$.
\end{lemma}

\begin{proof}
The proof goes as in the real valued case, see \cite{Billingsley}.
\\
By \cref{prop:right_continuity}, it suffices to prove that $\B(D(I:X))\subset \sigma(e_t:t\in T)$. Notice that by writing $e_0=\lim_{t_i\searrow 0}e_{t_i}$ for some sequence $(t_i)\subset T$, we obtain Borel measurability of $e_0$. Thus, we may assume that $0\in T$.
\\
The maps $e_{t_J}\coloneqq (e_{t_1},\dots,e_{t_n})\colon D(I:X)\to X^n$ are $(\sigma(e_t:t\in T), B(X^n))$-measurable by definition for all $t_J=(t_1,\dots, t_n)$ and $|J|\coloneqq n\in \N$. 
Moreover, for a partition $t_J$, $|J|=n$, of $[0,1]$, the map $\iota_{t_J}\colon X^n\to D(I:X)$, defined as $\iota_{t_J}(x)(t)\coloneqq x_i$ for $t\in[t_i,t_{i+1})$, $1\le i<n$, and $\iota_{t_J}(x)(1)=x_n$, is continuous, and therefore Borel measurable.
\\
Let now $(t_{J_n})$, $t_{J_n}\subset T$, be a sequence of partitions of $I$ with mesh $|t_{J_n}|$ going to zero. Then by above we have that the composition map $S_n\coloneqq \iota_{t_{J_n}}\circ e_{t_{J_n}}$ is $(\sigma(e_t:t\in T),\B(D(I:X)))$-measurable. Moreover, we have that $\id_{D(I:X)}=\lim_{n\to\infty}S_n$.
Thus the identity is $(\sigma(e_t:t\in T),\B(D(I:X)))$-measurable and hence $\B(D(I:X))\subset \sigma(e_t:t\in T)$, which concludes the proof.
\end{proof}

\subsection{Further auxiliary results}\label{sec:auxiliar}
Here we collect some results that are used later in the proof of main theorems.
The first statement is about the lower semi-continuity of pointwise variation, which follows immediately by combining \cref{lemma:topoL1-Sk} together with \cref{lemma:lscPBV} and \cref{lemma:normalizedBV} (and taking into account possible discontinuities at $t=1$):

\begin{lemma}\label{lma:varlsc}
The pointwise variation $\var\colon D(I:X)\to [0,\infty]$ is lower semi-continuous.
\end{lemma}

The next proposition concerns the identification in \eqref{eq:inclusion_BV}: 
\begin{proposition}[Borel Selection]\label{lemma:selection}
    For any $\gamma\in BV(I:X)$, we let $\tilde{\gamma}$ denote the Càdlàg-representative left-continuous at $t=1$.
    Then the selection map $T\colon BV(I:X)\subset L^1\to D(I:X)$, $\gamma\mapsto\tilde\gamma$, is a Borel map.
\end{proposition}
\begin{proof}
    By \cref{lemma:lscPBV} and \cref{lma:varlsc}, $BV(I:X)$ and $\mathcal{BV}(I:X)$ are Borel subsets of $\big(L^1(I:X),d_{L^1}\big)$ and $\big(D(I:X),d_{Sk}\big)$, respectively.
    Notice that by definition, the subset $\mathcal{BV}^{1}$ of all curves in $\mathcal{BV}$ which is left-continuous at $t=1$ is closed under the Skorokhod metric. So it suffices to show that the following bijection
    \begin{equation}
        \mathfrak{t}\colon (BV(I:X),d_{L^1}) \rightarrow  (\mathcal{BV}^{1},d_{Sk}), \quad \gamma \mapsto \tilde{\gamma},
    \end{equation}
    is a Borel map.
    Due to \cref{lemma:topoL1-Sk}, $\mathfrak{t}^{-1}$ is continuous and the assertion follows after applying \cite[Propsition 4.5.1]{srivastava1998course}.
\end{proof}

The following proposition is needed in the proof of \ref{itm:marginalprop} in \cref{thm:lift_BV}.
Then to prove the time-marginals condition reduces to prove it only on a dense set of times $t$.
\begin{proposition}\label{prop:right_continuity}
The evaluation map $e_t\colon D(I:X)\to X$ defined as $e_t(\gamma) \coloneqq \gamma_t$ is Borel measurable. Moreover, for any $\pi\in P(D(I:X))$ the function
\begin{align}
    t\mapsto \int \phi\circ e_t\d\pi
\end{align}
is Càdlàg for every continuous and bounded $\phi\in C_b(X:\R)$.
\end{proposition}

\begin{proof}
{\bf Step 1.} For $t\in\{0,1\}$, the map $e_t$ is actually continuous. Therefore it suffices to consider $t\in(0,1)$. Moreover, it suffices to prove that $\psi\circ e_t$ is Borel measurable for every continuous and bounded $\psi\in C_b(X)$. The rest of the proof goes in the lines of the real valued case, cf. \cite{Billingsley}; for $m\in \N$, define $\psi_m\colon D(I:X)\to \R$
\begin{align}
    \psi_m(\gamma)\coloneqq \frac1m\int_t^{t+\frac1m}\psi(\gamma(s))\d s.
\end{align}
We want to show that $\psi_m$ is continuous. Let $\gamma_n \to \gamma$ in $D(I:X)$. Then $\gamma_n(s)\to \gamma(s)$ $\mathcal{L}^1$-almost everywhere. Thus, by dominated convergence we have that 
\begin{align}
    \psi_m(\gamma_n)=\frac1m\int_t^{t+\frac1m}\psi(\gamma_n(s))\d s\to\frac1m\int_t^{t+\frac1m}\psi(\gamma(s))\d s=\psi_m(\gamma),
\end{align}
when $n\to \infty$. We conclude that $\psi_m$ is continuous. On the other hand, by right continuity of $\psi\circ \gamma$, we have that $\psi\circ e_t(\gamma)=\lim_{m\to\infty}\psi_m(\gamma)$. Hence, $\psi\circ e_t$ is Borel measurable.

{\bf Step 2.} Let us first prove the right continuity for all $t\in [0,1)$. For that, let $\phi\in C_b(X)$, and $C>0$ so that $|\phi|\le C$. Fix $t_0\in[0,1)$ and $\varepsilon>0$. Write 
\begin{align}
   \Gamma\coloneqq D(I:X)= \bigcup_{n\in\N}\Gamma_n,
\end{align}
where
\begin{align}\label{eq:Gamman}
    \Gamma_n\coloneqq \left\{\gamma\in D(I:X): \lvert\phi(\gamma(t_0))-\phi(\gamma(t))\rvert<\varepsilon, \forall t \in [t_0, t_0+\frac1n] \right\}.
\end{align}
This is possible since $\phi\circ \gamma$ is right continuous for every $\gamma\in D(I:X)$. Since $\Gamma_n$ is increasing in $n$, there exists $n_0\in\N$ so that 
\[\pi(\Gamma\setminus \Gamma_{n_0})\le \frac{\varepsilon}{C}.\]
Let $s\in[t_0,t_0+\frac1{n_0}]$. Then
\begin{align}\label{eq:rightcontinuity}
    \Big\lvert\int \phi\circ e_{t_0}-\phi\circ e_s\d\pi\Big\rvert&\le\Big\lvert\int_{\Gamma\setminus\Gamma_{n_0}} \phi\circ e_{t_0}-\phi\circ e_s\d\pi\Big\rvert+\Big\lvert\int_{\Gamma_{n_0}} \phi\circ e_{t_0}-\phi\circ e_s\d\pi\Big\rvert
    \\ &\le\int_{\Gamma\setminus\Gamma_{n_0}}\big\lvert \phi\circ e_{t_0}-\phi\circ e_s\big\rvert\d\pi+\int_{\Gamma_{n_0}}\big\lvert \phi\circ e_{t_0}-\phi\circ e_s\big\rvert\d\pi
    \\ &\le 2C\pi(\Gamma\setminus \Gamma_{n_0})+\varepsilon\le 3\varepsilon. 
\end{align}
Thus, the map $t\mapsto \int\phi\circ e_t\d\pi$ is right continuous.

The existence of left limits goes in the same lines. The only modifications needed are the following. First of all, in \eqref{eq:Gamman} the set $\Gamma_n$ is replaced by
\begin{align}
    \tilde\Gamma_n\coloneqq \left\{\gamma\in D(I:X): \lvert\phi(\gamma(t_0^-))-\phi(\gamma(t))\rvert<\varepsilon, \forall t\in[t_0-\frac1n,t_0]\right\}.
\end{align}
Second, in \eqref{eq:rightcontinuity}, the estimates are done for $t,s\in[t_0-\frac1{n_0},t_0)$ from which the existence of left limits follows by Cauchy's criterion. 
\end{proof}

Finally, the observation below will be useful when dealing with (non-continuous) BV-curves in the Wasserstein space.

\begin{lemma}\label{lma:precompactness}
Let $\gamma\in D(I:X)$. Then the image of $\gamma$ is precompact.
\end{lemma}
\begin{proof}
Let $C\subset X$ be a collection of all left and right limits of $\gamma$, and let $S\coloneqq C\setminus \mathrm{Im}(\gamma)$. Clearly $C$ is the closure of the image of $\gamma$. It suffices to prove that for any sequence $(x_i)\subset S$ there exists a subsequence that converges. Indeed, if $(x_i)\subset C$ so that there exists infinitely many $i$ for which $x_i=\gamma(t_i)$, then we can take a monotone subsequence of $t_i$ and by the existence of left and right limits conclude that this subsequence converges to a point in $C$.

Therefore, let $(x_i)\subset S$. For any $i$, choose $t_i$ so that $d(\gamma(t_i),x_i)<\frac1{2^i}$. Then, as before, there exists a monotone subsequence of $(t_i)$, still denoted by $t_i$. Let $a=\lim\gamma(t_i)\in C$. Then (for the corresponding subsequence of $(x_i)$) we have that
\begin{align}d(a,x_i)\le d(a,\gamma(t_i))+\frac1{2^i}\to 0,
\end{align}
when $i\to\infty$.
\end{proof}

\subsection{Equivalent definitions of BV-curves}\label{sec:EquiBV}
It is known, especially in the real-valued case (see, e.g., \cite{Leoni}), that different definitions of BV-curves are equivalent.
Although we expect that such a result is also known in the general setting of metric spaces, we did not easily find it in the literature.
Thus, for the sake of completeness, we here prove the equivalence in the general case. 

\begin{theorem}[Equivalent definitions of BV-curves]\label{thm:DefBVcurves}
  Let $(X,d)$ be a complete and separable metric space. Given $u \in L^1 (I : X)$, the following are equivalent:
  \begin{enumerate}
    \item\label{item:EqBV1} $u$ is a BV-curve;
    \item\label{item:EqBV2} the difference quotient $\Delta_h u(t)\coloneqq d(u(t+h),u(t))/h$ for $h\in (0,1)$ satisfies
    \begin{equation}\label{ineq:Lisini}
      \sup_{0 < h < 1} \int_0^{1 - h}\Delta_h u (t)\d t <\infty;
    \end{equation}
    \item\label{item:EqBV3} there is a finite Borel measure $\mu\in\mathcal{M}([0,1])$ such that for any $\varphi\in \Lip_1(X:\mathbb{R})$, $\varphi\circ u$ is a BV-function and
    \begin{align}\label{eq:Ambrosio}
    |D(\varphi\circ u) |\leq\mu ; 
    \end{align}
    \item\label{item:EqBV4} there is a finite Borel measure $\mu\in\mathcal{M}([0,1])$ such that for any $\varphi=d(\cdot,x)$ with $x\in X$, $\varphi\circ u$ is a BV-function and \eqref{eq:Ambrosio} holds.
    \item\label{item:EqBv5} there exists a finite Borel measure $\mu \in \mathcal{M} ([0, 1])$ such that for $\mathcal{L}^2$-a.e. $(s, t) \in I^2$
    \begin{equation}\label{ineq:DisLeqMeasure}
               d(u(t),u(s)) \leq \mu ([s, t]). 
               \end{equation}
 \end{enumerate}
\end{theorem}

\begin{proof}
  $\eqref{item:EqBV1} \Rightarrow \eqref{item:EqBV2}$. If $\essvar(u)<\infty$, since $\left(
  \ref{ineq:Lisini} \right)$ is stable under choosing different
  representatives, we can assume $u \in D (I : X)$ and $\var
  (u) = \essvar (u)$. Then for any $h > 0$, the function $h \mapsto
  \Delta_h u (t)$ is bounded and the set of discontinuity points is $\mathcal{L}^1$-negligible, making it Riemann integrable by Riemann-Lebesgue theorem.
  For any $n \in \mathbb{N}$, we divide $I$ into $0 = t_0 < \cdots t_i < \cdots <
  t_N \leq 1 - h$, where $N = \lfloor (1 - h) n / h \rfloor$ and $t_i = i h /
  n$. Using triangle inequality,
  \begin{equation}
    h \Delta_h u (t_i) = d (u (t_i), u (t_i + h)) \leq \sum_{j = 0}^{n - 1} d
    (u (t_{i + j}), u (t_{i + j + 1})), \quad 1 \leq i \leq N - 1.
    \label{ineq:tria-Riesum}
  \end{equation}
  So for any $h > 0$, with convention $t_{N + 1} \coloneqq 1 - h$, we have
  \begin{eqnarray*}
    \int_0^{1 - h} \Delta_h u (t) \d t & = & \lim_{n \rightarrow \infty}
    \sum_{i = 0}^N \Delta_h u (t_i) (t_{i + 1} - t_i)\\
    & = & \lim_{n \rightarrow \infty} \left( \Delta_h u (t_N) (1 - h - t_N) +
    \frac{h}{n} \sum_{i = 0}^{N - 1} \Delta_h u (t_i) \right)\\
    & \leq & \liminf_{n \rightarrow \infty} \frac{1}{n} \left( \sum_{i =
    0}^{N - 1} \sum_{j = 0}^{n - 1} d (u (t_{i + j}), u (t_{i + j + 1})) + d
    (u (t_N), u (1 - h)) \right)\\
    & \leq & \sum_{i = 0}^N d (u (t_i), u (t_{i + 1})) \leq \var (u) .
  \end{eqnarray*}
  $\eqref{item:EqBV2}\Rightarrow \eqref{item:EqBV3}$. Let $L$ be the following non-decreasing function
  \begin{equation}\label{eq:LSmeasureV}
  L (t) = \sup_{0 < h < t} \int_0^{t - h} \Delta_h u (s) \d s 
  \end{equation}
  and $| D L |$ be the Lebesgue-Stieltjes measure such that $| D L | ((a, b]) = L
  (b^+) - L (a^+) .$ By assumption $| D L |$ is a finite measure and \cref{lemma:TmeasureL} below ensures
  \[ 
  \sup_{0< h<b-a}\int_a^{b-h} \Delta_h u(t)\d t\leq|D L |([a,b]).
  \]
Given $\varphi \in \Lip_1(X:\mathbb{R})$,
  \begin{equation}
    \sup_{0 < h < b-a} \int_a^{b - h} \Delta_h (\varphi \circ u) (s) \d s \leq
    \sup_{0 < h < b - a} \int_a^{b - h} \Delta_h u (t) \d t \leq | D L|([a,b]) . \label{ineq:SobolevTV}
  \end{equation}
  Using \cite[Corollary 2.43]{Leoni}, $\varphi \circ u$ is a BV-function on $I$ and
  over any interval $(a, b)$
  \begin{equation}\label{eq:W11toVarM,Realfct}
  \essvar(\varphi\circ u;(a, b))=\sup_{0<h< b-a}\int_a^{b-h} \Delta_h (\varphi \circ u) (s) \d s.
  \end{equation}
  Thus, with \eqref{ineq:SobolevTV},
  \begin{eqnarray*}
    | D (\varphi \circ u) | ((a, b)) & = & \lim_{r\searrow0} | D (\varphi\circ u) | ((a+r,b-r))\\
    & \leq &\liminf_{r\searrow0} | D L | ([a+r,b-r])\\
    & \leq & | D L | ((a, b))
  \end{eqnarray*}
  for arbitrary $\varphi \in \Lip_1 (X : \mathbb{R})$ and $(a, b)
  \subset I$. $\eqref{item:EqBV3}\Rightarrow \eqref{item:EqBV4}$ is trivial.\\
  $\eqref{item:EqBV4} \Rightarrow \eqref{item:EqBV1}$ Let $\{z_i\}_{i \in \mathbb{N}} \subset X$ be a dense set.
  Define the functions $\varphi_i (\cdot) \coloneqq d(z_i, \cdot)$.
  By assumption $\varphi_i\circ u \in BV(I:\mathbb{R})$.
  Thus for each $i$, we can take a representative $\phi_i \in D(I:\mathbb{R})$ and $N_i\subset I$ such that $\mathcal{L}^1(N_i) = 0$ and $\phi_i = \varphi_i\circ u$ on $I \setminus N_i$.
  Set $\tilde{I} = I \setminus \cup_{i \in \mathbb{N}} N_i$. 
  Our goal is to find a representative $\tilde{u}$ whose pointwise variation is finite. 
  
  Notice that by the density of the set $\{z_i\}$ in $X$, we have
  \begin{align}
      d(u(s),u(t))
      & = \sup_i \big| d(z_i,u(s)) - d(z_i,u(t)) \big|.
  \end{align}
    Further using the assumption on measure $\mu$, for $s,t \in \tilde{I} $ and $s<t$, we have
    \begin{align}
      d(u(s),u(t))  = \sup_i \big| \varphi_i(s) - \varphi_i(t) \big| \leq \sup_i |D \varphi_i|((s,t]) \leq \mu((s,t]). 
  \end{align}
  Now take a decreasing sequence $t_i \rightarrow t$ with $t_i \in \tilde{I}$
  \begin{align}
      \sum_{i=k}^{\infty} d (u(t_i),u(t_{i+1}))
      & \leq \sum_{i=k}^{\infty} \mu( (t_{i+1},t_{i}]) \leq \mu( (t,t_{k}]).
  \end{align}
  Since $\mu$ is a finite measure, the right hand side in equation above goes to 0 as $k \rightarrow \infty $. Therefore, $(u(t_i))_{i \in \mathbb{N}}$ is a Cauchy sequence and by completeness, the limit 
  \begin{equation}
     \tilde{u} (t) \coloneqq \lim_{i \rightarrow \infty} u(t_i)
  \end{equation}
  exists and we have $\tilde{u}= u$ over $\tilde{I}$. Finally, by the similar argument as in the proof of Lemma \ref{lemma:normalizedBV}, equation \eqref{eq:finitevar}, we conclude
  \begin{equation}
      \var (\tilde{u}) \leq \mu (I) < \infty.
  \end{equation}
$\eqref{item:EqBv5} \Rightarrow \eqref{item:EqBV3}$. Let $\varphi:X\rightarrow \R$ be any $1$-Lipschitz function. As assumed, 
\begin{equation}
|\varphi_u(t)-\varphi_u(s)|\leq \mu([s,t])\quad \varphi_u\coloneqq \varphi\circ u
\end{equation}
for $\L^2$-a.e. $(s,t)\in I^2$. Then there is $H\subset I$ of full $\mathcal{L}^1$ measure such that for any $h\in H$, $ |\varphi_u(t)-\varphi_u(t+h)|\leq\mu([t,t+h])$ over $\mathcal{L}^1$-a.e. $t\in(0,1-h)$.
Since the monotone function $t\mapsto \mu ([0, t])$ is discontinuous at most at countably many points, $\mu([0, t])=\mu ([0, t))$ for $\mathcal{L}^1$-a.e. $t \in (0, 1)$. Now for every $0\leq a<b\leq1$ and $h\in H$,
  \begin{align}
    \int_a^{b- h} \Delta_h \varphi_u (t) \d t & \leq  \int_a^{b-h} \frac{\mu ([t, t
    + h])}{h} \d t\\
    & =  \frac{1}{h} \int_a^{b - h} \mu ([0, t + h]) - \mu ([0, t)) \d t\\
    & =  \frac{1}{h} \left( \int_{a+h}^b \mu ([0, t]) \d t - \int_a^{b - h} \mu
    ([0, t]) \d t \right)\\
    & =  \frac{1}{h} \left( \int_{b- h}^b \mu ([0, t]) dt - \int_a^{a+h} \mu
    ([0, t]) \d t \right)\\
    &\leq\mu([0,b])-\mu([0,a))=\mu ([a, b])<\infty \label{ineq:LisiniLeqMeasure}.
  \end{align}
 Moreover, we can extend  \eqref{ineq:LisiniLeqMeasure} to those $h\notin H$ using the continuity given by \cref{lemma:TmeasureL} \eqref{item:Lemma2.10.2}.
 With again the relation \eqref{eq:W11toVarM,Realfct}, we know $\varphi_u$ is a BV-function and $|D\varphi_u|\leq\mu$.\\
  $\eqref{item:EqBV1} \Rightarrow \eqref{item:EqBv5}$. It suffices to prove the statement for $u \in
  {D} (I:X)$ satisfying $\left( \ref{eq:PBVachieved} \right)$ as in
  Lemma \ref{lemma:normalizedBV}. By choosing $\mu$ as the variation measure $|Du|$, it follows
  \begin{eqnarray*}
    d(u(s), u(t))&\leq&\liminf_{r\searrow0}\var(u;(s-r,t+r))\\
    & = & \liminf_{r\searrow0}\essvar(u;(s-r,t+r))=\mu([s,t]).
  \end{eqnarray*}
\end{proof}
\begin{lemma}\label{lemma:TmeasureL}
For any $u\in L^1(I:X)$ and $0\leq a<b\leq 1$, define $L$ as in \eqref{eq:LSmeasureV} and 
\begin{equation}
l^b_a(h)\coloneqq \int_a^{b-h} \Delta_h u(t)\d t,\quad 0<h<b-a.
\end{equation}
Then 
\begin{enumerate}
\item\label{item:Lemma2.10.1} for all $0<h<b-a$, $l^b_a(h)\leq l^b_a(h/2)$, and in particular, 
\[
\sup_{0<h<b-a}l^b_a(h)=\limsup_{h\rightarrow 0}l^b_a(h);
\]
\item\label{item:Lemma2.10.2} if $X=\R$, then $(0,b-a)\ni h\mapsto l^b_a(h)$ is continuous;
\item\label{item:Lemma2.10.3} for general metric space $X$, if $L(1)<\infty$, $(0,b-a)\ni h\mapsto l^b_a(h)$ is continuous and
\begin{equation}\label{ineq:W1,1toVarM}
  \sup_{0< h<b-a}\int_a^{b-h} \Delta_h u(t)\d t \leq | D L |([a,b]).
\end{equation}
\end{enumerate}
\end{lemma}
\begin{proof}
The first assertion follows simply by triangle inequality:
\begin{align}
 \int^{b-h}_{a}\frac{d(u(t),u(t+h))}{h}\d t&\leq \int^{b-h}_{a}\frac{d(u(t),u(t+h/2))+d(u(t+h/2),u(t+h))}{h}\d t\\
 &\leq \int^{b-h/2}_{a}\frac{d(u(t),u(t+h/2))}{h/2}\d t.
\end{align}
Given arbitrary $h,h'$, by triangle inequality
\[
\int |d(u(t),u(t+h))-d(u(t),u(t+h'))|\d t \leq \int d(u(t+h),u(t+h'))\d t.
\]
So the continuity of $l^b_a$ at $h$ boils down to show
\begin{equation}\label{eq:translationConti}
    \lim_{h\rightarrow 0}\int d(u(t),u(t+h))\d t=0.
\end{equation}
If $X=\R$, \eqref{eq:translationConti} results from the fact that any $L^1$-function can be approximated by continuous functions under $L^1$-norm. On general metric spaces, \eqref{eq:translationConti} immediately follows from the finiteness of $L(1)$.\\ 
As for \eqref{ineq:W1,1toVarM}, by definition, $|DL|([a,b])=L(b^+)-L(a^-)$, so it suffices to show $l^b_a(h_1)+l^a_0(h_2)\leq L(b)$ for all $h_1,h_2$. By the continuity of $h\mapsto l^b_a(h)$, we can assume
\[
h_i=\frac{k_i}{2^n},\quad n,k_i\in\N, 1\leq k_i<2^n,i=1,2.
\]
The conclusion follows if we take the uniform step-size $1/2^n$ and argue as in part $(1)$:
\begin{align}
   l^b_a(h_1)+l^a_0(h_2)&= \int^{b-h_1}_{a}\frac{d(u(t),u(t+h_1))}{h_1}\d t+ \int^{a-h_2}_{0}\frac{d(u(t),u(t+h_2))}{h_2}\d t\\
   &\leq\int_a^{b-h_1}\sum_{k=1}^{k_1}d\Big(u\big(t+\frac{k-1}{2^n}\big),u\big(t+\frac{k}{2^n}\big)\Big)\frac{2^n}{k_1}\d t\\
   &\quad+ \int_0^{a-h_2}\sum_{k=1}^{k_2}d\Big(u\big(t+\frac{k-1}{2^n}\big),u\big(t+\frac{k}{2^n}\big)\Big)\frac{2^n}{k_2}\d t\\
   &\leq \int_0^{b-1/2^n}\frac{d\big(u(t),u(t+1/2^n)\big)}{1/2^n}\d t\leq L(b).
\end{align}
\end{proof}
\begin{rem}[Equivalence of variation measures]\label{rem:equalvariation}
    In \cref{thm:DefBVcurves}, we show that five different definitions of the \emph{set} of BV-curves are in fact equivalent. 
    The proof has an even stronger implication, namely, that the five(-ish) different notions of the variation measure of a BV-curve are equal. 
    To be more precise, given a BV-curve $u$, the following measures are equal:
    \begin{itemize}
        \item[(\hspace{0.5pt}0\hspace{2.1pt})] the Lebesgue--Stieltjes  measure induced by $\essvar(u;(0,\cdot))$;
        \item[(\hspace{0.5pt}1\hspace{2.1pt})] the measure $\mu$ characterized by $\mu((a,b))=\essvar(u;(a,b))$;
        \item[(\hspace{0.7pt}2\hspace{1.8pt})] the measure $\mu$ characterized by $\mu((a,b))=\sup_{h}\int_a^{b-h}\Delta_h u (t)\d t$;
        \item[(\hspace{0.5pt}3\hspace{2.1pt})] the measure $\mu=\underset{\varphi\in \Lip_1(X)}{\mathcal{M}-\sup}|D(\varphi\circ u)|$; \footnote{For the general construction of a supremum measure over an arbitrary collection of measures, see \cite{Ambrosio1990MetricSV}. The bound \eqref{eq:Ambrosio} ensures that the construction yields a finite measure, and indeed it is the minimal such a bound $\mu$. } 
        \item[(3$'$)] the minimal measure satisfying \eqref{eq:Ambrosio};
       
        \item[(\hspace{0.5pt}4\hspace{2.1pt})]the minimal measure satisfying \eqref{eq:Ambrosio} for all $\varphi$ of the form $\varphi=d(\cdot,x)$;
        \item[(\hspace{0.5pt}5\hspace{2.1pt})] the minimal measure satisfying \eqref{ineq:DisLeqMeasure}.
    \end{itemize}
    
    The fact that all these different approaches lead to the same measure is evident from the corresponding steps in the proof of \cref{thm:DefBVcurves}.
    The single measure obtained in one of the various ways is thus denoted by $|Du|$ and called the variation measure of $u$.
    Moreover, we take advantage of the different approaches without mentioning it explicitly.
\end{rem}
\begin{rem}\label{rem:NEWEquiV1&4}
Sometimes it is useful to have bounds as above which hold not only for almost every $s$ and $t$, but in fact everywhere.
For this one needs to consider a specific choice for a representative of the BV-curve in question.
One choice, which we decided to work with in this paper, is the Càdlàg-representative.
Since $D(I:X)$-representative of a BV-curve is unique (up to its value at $t=1$),  if $u\in \mathcal{BV}(I:X)$, by Lemma \ref{lemma:normalizedBV}, we indeed have that $d(u(s),u(t))\leq|Du|((s,t])$ for \emph{all} $0\leq s\leq t<1$.
\end{rem}

\section{Main results}\label{sec:results}
\subsection{Lifts of AC- and BV-curves in 1-Wasserstein spaces} As introduced in \cref{sec:introduction,subsec:notation}, we consider  $(X,d)$  a complete and separable metric space and $P_1(X)$ the associated Wasserstein space of order $p=1$. Without loss of generality, we fix the  time interval to be $I=[0,1]$.   
\begin{theorem}\label{thm:BVSk}
    Let $\pi \in P(D(I:X))$ be concentrated on $\mathcal{BV}(I:X) \subset D(I:X)$ such that
    \begin{equation}\label{eq:Dmu_inequality}
        \int | D \gamma | (I) \d \pi (\gamma) < \infty 
    \end{equation}
    and $\mu_0  \coloneqq (e_0)_\# \pi \in P_1(X)$. Then the curve $t\mapsto\mu_t \coloneqq (e_t)_\# \pi$ belongs to $\mathcal{BV}(I: P_1(X))$, and 
    \begin{equation}\label{eq:easyineq}
        |D \mu | \leq \int |D \gamma| \d\pi (\gamma)
    \end{equation}
    as measures.
\end{theorem}
\begin{rem}\label{rem:Dgammapi}
    We emphasize that the right-hand side of \eqref{eq:easyineq} is a short-hand notation for the measure which is defined by
    \begin{equation}\label{eq:integralofmeasure}
        \left( \int_{D(I:X)} |D \gamma| \d\pi(\gamma) \right) (A) \coloneqq \int_{D(I:X)} |D \gamma| (A) \d\pi(\gamma) 
    \end{equation}
    for any Borel set $A \subset I$. 
    Notice that for an open set $U$, the map $\gamma \mapsto |D \gamma|(U)$ is Borel measurable due to the lower semi-continuity of variation (\cref{lma:varlsc}). 
    Then by standard measure theory techniques, $\gamma \mapsto |D \gamma|(A)$ is also Borel measurable, hence the set function in \eqref{eq:integralofmeasure} is well-defined, finite (by \eqref{eq:Dmu_inequality}), and actually a measure. 
    Moreover, the integral of any (non-negative) Borel function $f\colon I \to \mathbb{R}$ with respect to this measure is given by $\iint f(t) \d |D \gamma | (t) \d \pi (\gamma)$. 
\end{rem}

\begin{proof}[Proof of \cref{thm:BVSk}]
By Lemma \ref{lemma:normalizedBV},  $\pi$-a.e. $\gamma$ has bounded variation and over each interval $(a,b)\subset I$, we have that $|D\gamma|((a,b))=\essvar (\gamma;(a,b))=\var(\gamma;(a,b))$.
Fix a point $\bar{x}\in X$.
Using the fact that $\mu_t = (e_t)_\# \pi$, we have for all $t\in (0,1)$ that
\begin{align}
\int_X d (\bar{x},x) \d\mu_t(x)&=\int_{D(I:X)} d(\bar{x},\gamma(t)) \d\pi(\gamma)\\
&\leq  \int_{D(I:X)} \Big[  d(\bar{x},\gamma(0))+d(\gamma(0),\gamma(t))\Big] \d\pi(\gamma)\\
&\leq W_1(\delta_{\bar{x}},\mu_0) + \int_{D(I:X)} |D\gamma|([0,1]) \d\pi(\gamma)<\infty.\label{ineq:pfBVsk1}
\end{align}
In particular, $\mu_t\in P_1(X)$ for all $t\in[0,1)$. 
At each $t\in[0,1)$, by right-continuity of curves in $D(I:X)$ and dominated convergence theorem,
 \[
  W_1(\mu_t,\mu_{t+r})\leq\int_{D(I:X)} d(\gamma(t),\gamma(t+r))\d\pi(\gamma)\rightarrow 0
\]
as $r\rightarrow 0$. 
Similarly, we can show $t\mapsto \mu_t$ has left limit in $P_1(X)$ at each $t\in(0,1]$.  In other words, $(\mu_t)\in D(I:P_1(X))$ and in particular with \eqref{ineq:pfBVsk1}, $t\mapsto \mu_t\in L^1(I:P_1(X))$. Now for arbitrary $0\leq s<t\leq1$, we can estimate
\begin{align}
W_1(\mu_t,\mu_s)&\leq \int_{D(I:X)} d(\gamma(t),\gamma(s))\d\pi(\gamma)\\
&\leq \int_{D(I:X)} |D\gamma|([s,t])\d\pi(\gamma)= \left(  \int_{D(I:X)} |D\gamma|\d\pi(\gamma)\right)([s,t]),
\end{align}
where the second inequality follows from Remark \ref{rem:NEWEquiV1&4}. Theorem \ref{thm:DefBVcurves} ensures that $(\mu_t)\in \mathcal{BV}(I:P_1(X))$, and by \cref{rem:equalvariation}, we have that 
\[
|D\mu|\leq\int_{D(I:X)} |D\gamma|\d\pi(\gamma).
\]
This concludes the proof.
\end{proof}

The previous theorem states that any lift $\pi$ of a BV-curve $(\mu_t)$ provides an upper bound for the variation measure of $(\mu_t)$ through equation \eqref{eq:easyineq}.
In the next theorem, a measure $\tilde{\pi}$ is constructed, using techniques of optimal transportation, that achieves equality and thus entails a key relation on the variation of the curves.
Our main result is the following:

\begin{theorem}[AC- and BV-curves in 1-Wasserstein spaces]\label{thm:lift_BV}
    Let $(\mu_t)\in \mathcal{BV}(I:P_1(X))$.
    Then there exists a probability measure $\tilde{\pi}\in P(D(I:X))$ such that
    \begin{enumerate}[label=(\roman*), font=\normalfont]
        \item\label{itm:support} $\tilde{\pi}$ is concentrated on $\mathcal{BV}(I:X)\subset D(I:X)$;
        \item\label{itm:marginalprop} $(e_t)_\#\tilde{\pi}=\mu_t$ for all $t\in I$;
        \item\label{itm:totalvarian} the total variation measure $|D\mu|$ satisfies\footnote{
        Two comments are provided: (1) As usual, we interpret the measure on the right-hand side of \eqref{eq:Dmu}  in accordance with \cref{rem:Dgammapi}.
        (2) The measure $\tilde{\pi}$ is constructed in a way so that an equation of the form \eqref{eq:Dmu} holds for the pointwise variation on the whole interval $I$. For more details, see Step 4 in the proof.}
       \begin{equation}\label{eq:Dmu}
           |D\mu|=\int|D\gamma|\d\tilde{\pi}(\gamma).
       \end{equation}
    \end{enumerate}
    Moreover, the absolutely continuous part $|\dot\mu|\mathcal{L}^1$ of the measure $|D\mu|$, given by the metric derivative, satisfies
    \begin{align}\label{eq:dotmu_BV}
         |\dot\mu_t|   = \lim_{h \rightarrow 0 } \int \frac{  d (\gamma_{t}, \gamma_{t+h})}{|h|} \d\tilde{\pi}(\gamma) = \lim_{h \rightarrow 0 } \int \frac{ |D \gamma|([t,t+h])}{|h|} \d\tilde{\pi}(\gamma)
    \end{align}
    for $\mathcal{L}^1$-a.e. $t\in I$. 
    \\
    In particular, if $(\mu_t) \in \mathcal{AC}^1(I:P_1(X))$, then $|D\mu|=|\dot\mu|\mathcal{L}^1$ and the metric speed $|\dot\mu_t|$ is characterized by the equation above. 
\end{theorem}

\begin{proof}\label{proof:lift_BV}
    Our proof firmly follows the one of \cite[Theorem 5]{Lisini} with modifications for BV-curves established in Section \ref{sec:preliminaries}.
    \\ For any integer $N\in \N$, we divide $I$ into $2^N$ pieces and denote $t^i\coloneqq i/2^N$ for $i=0,\cdots,2^N$. Let $X_i$ represent the $i$-th copy of $X$ and take the product space
    \begin{equation}
        \textbf{X}_{N} \coloneqq X_0\times X_1\times \cdots \times X_{2^N}.
    \end{equation}
     It is always possible (see e.g. \cite[Section 5.3]{AGSmetricflow}) to find $\eta_N\in P(\textbf{X}_N)$ s.t.
 \[
 \mathrm{Pr}^i_{\#}\eta_N=\mu_{t^i},\quad \mathrm{Pr}^{i,i+1}_{\#}\eta_{N}\in \mathrm{Opt}(\mu_{t^i},\mu_{t^{i+1}}),
 \]
 where $\mathrm{Opt}(\mu_{t^i},\mu_{t^{i+1}})$ is the set of optimal couplings between $\mu_{t^{i}}$ and $\mu_{t^{i+1}}$ and the maps $\mathrm{Pr}^i$, $\mathrm{Pr}^{i,j}$ are projections from $\textbf{X}_{N}$ to the $i$-th, $(i,j)$-th component, respectively. Finally, we define the filling map $\sigma:\textbf{x}=(x_0,\cdots ,x_{2^N})\in\textbf{X}_N\mapsto \sigma_x\in L^1(I:X)$ by 
 \[
 \sigma_x(t)\coloneqq x_i,\quad t\in[t^i,t^{i+1}); \quad \sigma_{\textbf{x}}(1)\coloneqq x_{2^N};
 \]
 and set $\pi_N\coloneqq \sigma_{\#}\eta_N\in P(L^1(I:X))$.\\
 \textbf{Step 1 (Tightness of $\{\pi_N:N\in\N\}$).} It is known that tightness of $\{\pi_N:N\in\N\}$ is equivalent to the existence of a function $\Phi:L^1(I:X)\rightarrow [0,\infty]$ whose sublevels $\lambda_c(\Phi)\coloneqq\{u\in L^1(I:X):\Phi(u)\leq c\}$ are compact in $L^1(I:X)$ for any $c\in\R_{+}$ and 
 \begin{equation}\label{ineq:TightnessPhi}
 \sup_{N\in\N} \int_{L^1(I:X)}\Phi(u)\d\pi_N(u)<\infty.
 \end{equation} 
 Clearly, $\{\mu_t:t\in[0,1]\}$ is bounded in $P_1(X)$, so for fixed $\bar{x}\in X$ 
 \begin{equation}
 C_1\coloneqq\sup_{t\in I}\int_{X}d(x,\bar{x})\d \mu_t(x)<\infty.
 \end{equation}
 From \cref{lma:precompactness}, $\{\mu_t:t\in[0,1]\}$ is precompact in $P_1(X)$, so by Prokhorov's theorem it is tight which means there is a $\psi:X\rightarrow [0,\infty]$ whose sublevels $\lambda_c(\psi)\coloneqq\{x\in X:\psi(x)\leq c\}$ are compact in $X$ for all $c\in\R_{+}$ and 
 \begin{equation}
 C_2\coloneqq \sup_{t\in I} \int_{X}\psi(x)\d \mu_t(x)<\infty.
 \end{equation}
We claim that the function 
\begin{equation}
    \Phi(u)\coloneqq \int^1_0d(u(t),\bar{x})\d t+ \int^1_0\psi(u(t))\d t+ \sup_{0<h<1}\int^{1-h}_0\frac{d(u(t),u(t+h))}{h}\d t,
\end{equation}
on $L^1(I:X)$ satisfies \eqref{ineq:TightnessPhi}.\\
Firstly, for each $\pi_N$, 
\begin{align}
\int_{L^1(I:X)}\int^1_0d(u(t),\bar{x})+\psi(u(t))\d t\d \pi_N &= \int^1_0\int_{\textbf{X}_N}d(\sigma_{\textbf{x}}(t),\bar{x})+\psi(\sigma_{\textbf{x}}(t))\d \eta_N(\textbf{x})\d t\\
&= \sum^{2^N-1}_{i=0}\int^{t^{i+1}}_{t^i}\int_X d(x,\bar{x})+\psi(x)\d \mu_{t^i}(x)\d t\\
&= \frac{1}{2^N}\sum_{i=0}^{2^N-1}\int_Xd(x,\bar{x})+\psi(x)\d \mu_{t^i}(x) \leq C_1+C_2.
\end{align}Secondly, from \cref{lemma:TmeasureL}, for every $u=\sigma_{\textbf{x}}$, $\textbf{x}=(x_0,x_1,...,x_{2^N})$, we have
\begin{align}
\sup_{0<h<1}\int^{1-h}_0\frac{d(u(t),u(t+h))}{h}\d t&=\sup_{0<h<1/2^N}\int^{1-h}_0\frac{d(u(t),u(t+h))}{h}\d t\\
&=\sup_{0<h<1/2^N}\sum^{2^N-2}_{i=0}\int^{t^{i+1}}_{t^{i+1}-h}\frac{d(\sigma_{\textbf{x}}(t),\sigma_{\textbf{x}}(t+h))}{h}\d t\\
&= \sum^{2^N-2}_{i=0}d(x_i,x_{i+1}).
\end{align}
Integrating the above equality over $\pi_N$, one has
\begin{align}
\int_{L^1(I:X)}\sup_{0<h<1}\int^{1-h}_0\frac{d(u(t),u(t+h))}{h}\d t\d\pi_N&=\int_{L^1(I:X)}\sum^{2^N-2}_{i=0}d(x_i,x_{i+1})\d\pi_N\\
&=\sum^{2^N-2}_{i=0}\int_{L^1(I:X)}d(x_i,x_{i+1})\d\pi_N\\
&=\sum^{2^N-2}_{i=0}W_1(\mu_{t^{i}},\mu_{t^{i+1}}) \\
&\leq\var(\mu;[0,1))\leq |D\mu|([0,1]),\label{ineq:pf2Thm3.3}
\end{align}
where the last inequality follows from the fact that $(\mu_t)\in D(I:P_1(X))$ (see \cref{lemma:normalizedBV}). \\
Combining two above estimates, we obtain
\begin{equation}
 \sup_{N\in\N} \int_{L^1(I:X)}\Phi(u)\d\pi_N(u)\leq C_1+C_2+|D\mu|([0,1])<\infty,
\end{equation} 
 which proves \eqref{ineq:TightnessPhi}.
 The precompact criterion in \cite[Theorem 2]{Lisini} guarantees all $\{\Phi\leq c\}$ are precompact. For the tightness of $\{\pi_N\}$, it remains to show that all sublevels of $\Phi$ are closed in $L^1(I:X)$. It suffices to prove $\Phi$ is lower semi-continuous with respect to $L^1$-convergence, which is a consequence of Fatou's Lemma. Indeed, given any $u_n\rightarrow u$ in $L^1$ (and it is not restrictive to assume further $u_n(t)\rightarrow u(t)$ for $\L^1$-a.e. $t\in I$), we have
 \begin{align}
 \sup_{0<h<1}\int^{1-h}_{0}\Delta_hu(t)\d t&= \sup_{0<h<1}\int^{1-h}_{0}\liminf_{n\rightarrow\infty}\Delta_hu_n(t)\d t\\
 &\leq \sup_{0<h<1}\liminf_{n\rightarrow\infty}\int^{1-h}_{0}\Delta_hu_n(t)\d t\\
 &\leq \liminf_{n\rightarrow\infty}\sup_{0<h<1}\int^{1-h}_{0}\Delta_hu_n(t)\d t.
 \end{align}
 In conclusion, by Prokhorov's theorem, there exists $\pi\in P(L^1(I:X))$ and a subsequence $N_k$ such that  $\pi_{N_k}\rightarrow \pi$  narrowly in $P(L^1(I:X))$ as $k\rightarrow \infty$.\\
 \textbf{Step 2 ($\pi$ is concentrated on $BV(I:X)$).} 
 As shown in the end of Step 1, the function 
 \[
 L^1(I:X)\ni u\mapsto \sup_{0<h<1}\int^{1-h}_{0}\Delta_hu(t)\d t
 \]
 is lower semi-continuous and bounded from below. So by narrowly convergence of $\pi_{N}$
 \begin{align}
     \int_{L^1(I:X)}\left(\sup_{0<h<1}\int^{1-h}_{0}\Delta_hu(t)\d t\right)\d \pi&\leq  \liminf_{N\rightarrow \infty}\int_{L^1(I:X)}\left(\sup_{0<h<1}\int^{1-h}_{0}\Delta_hu(t)\d t\right)\d \pi_N\\
     &\leq |D\mu|([0,1])<\infty,\label{pf3Theorem3.3}
 \end{align}
 where the second inequality comes from \eqref{ineq:pf2Thm3.3}.
 Therefore, 
 \[
 \sup_{0<h<1}\int^{1-h}_{0}\Delta_hu(t)\d t<\infty,\quad \text{for $\pi$-a.e. $u\in L^1(I:X)$}.
 \]
 By \cref{thm:DefBVcurves}, $\pi$ is concentrated on the Borel subset $BV(I:X)\subset L^1(I:X)$. Considering push-forward via the Borel selection map $T:BV(I:X)\subset L^1\rightarrow D(I:X)$ in \cref{lemma:selection}, we can construct the probability measure
 \[
 \tilde{\pi}\coloneqq T_{\#}\pi\in P(D(I:X)),
 \]
 which is concentrated on $\mathcal{BV}(I:X)$.\\
 \textbf{Step 3 (Proof of \ref{itm:marginalprop} and \ref{itm:totalvarian}).} Recall that for any BV-function $u$, 
 \begin{equation}
      |Du|((s,t))=\sup_{0< h<t-s}\int_s^{t-h} \Delta_h u(t)\d t,\quad 0\leq s<t\leq 1.
 \end{equation}
 Then we can repeat Step 1 to produce \eqref{pf3Theorem3.3} on each subinterval $[s,t]\subset I$:
 \begin{align}
     \int_{D(I:X)}|D\tilde{u}|\d\tilde{\pi}(\tilde{u})((s,t))&=\int_{L^1(I:X)}|D(T(u))|((s,t))\d\pi(u)\\
     &=\int_{L^1(I:X)}\sup_{0< h<t-s}\int_s^{t-h} \Delta_h u(t)\d t\d \pi(u)\\
     &\leq\liminf_{N\rightarrow \infty}\int_{L^1(I:X)}\left(\sup_{0<h<t-s}\int^{t-h}_{s}\Delta_hu(t)\d t\right)\d \pi_N\\
     &\leq |D\mu|([s,t]).
 \end{align}
Together with \cref{thm:BVSk}, \ref{itm:totalvarian} will be proved after obtaining \ref{itm:marginalprop}.\\
At last, for \ref{itm:marginalprop}, fix any test functions $\varphi\in C_b(X)$ and $\xi\in C_b(I)$. By noticing $u\mapsto \int_{[0,1]}\xi(t)\varphi(u(t))\d t$ is continuous on $L^1(I:X)$ and $(e_t)_{\#}\pi_N=\mu_{t^i}$, for each $N\in\N$, $t\in [t^i,t^{i+1})$, we have
\begin{align}
  \int_{D(I:X)}\int_0^1\xi(t)\varphi(u(t))\d t \d \tilde{\pi}(u) &=\int_{L^1(I:X)}\int_0^1\xi(t)\varphi(u(t))\d t \d \pi(u) \\
  &=\lim_{N\rightarrow\infty}\int_{L^1(I:X)}\int_0^1\xi(t)\varphi(u(t))\d t \d \pi_{N}(u)\\
  &=\lim_{N\rightarrow\infty}\int_0^1\xi(t)\int_{L^1(I:X)}\varphi(u(t)) \d \pi_{N}(u)\d t\\
  &=\lim_{N\rightarrow\infty}\sum^{2^N-1}_{i=0}\int^{t^{i+1}}_{t^i}\xi(t)\int_X\varphi(x)\d (e_t)_{\#}\pi_N(x)\d t\\
  &=\lim_{N\rightarrow\infty}\sum^{2^N-1}_{i=0}\int^{t^{i+1}}_{t^i}\xi(t)\int_X\varphi(x)\d \mu_{t^i}(x)\d t\\
  &= \lim_{N\rightarrow\infty}\sum^{2^N-1}_{i=0}\frac{1}{2^N}\xi(t^i)\int_X\varphi(x)\d \mu_{t^i}(x),\label{ineq:Riemsum}
\end{align}
where the last limit is guaranteed by continuity and boundedness of $\xi$ and $\varphi$.\\
Since $t\mapsto \mu_t\in D(I:P_1(X))$, the function
\[
t\mapsto \xi(t)\int_X\varphi(x)\d\mu_t(x)
\]
is in $D(I:\R)$ so it is continuous outside a set of countably many points, and in particular is Riemannian integrable. As a result, the limit of Riemann sums in \eqref{ineq:Riemsum} is equal to $\int\int\xi(t)\varphi(x)\d \mu_t\d t$. The arbitrariness of $\xi$ means
\begin{equation}\label{eq:pfThm3.3}
 \int_{D(I:X)}\varphi(u(t))\d \tilde{\pi}(u)=\int\varphi(x)\d \mu_t
\end{equation}
for $\L^1$-a.e. $t\in[0,1]$. 
By \cref{prop:right_continuity}, the function
\begin{equation}
    t\mapsto \int_{D(I:X)} \varphi(u(t))\d\tilde{\pi}(u)
\end{equation}
is right-continuous.
Therefore, \eqref{eq:pfThm3.3} holds for all $t\in[0,1)$ and $\varphi\in C_b(X)$, which implies $(e_t)_\#\tilde{\pi}=\mu_t$.\\
\textbf{Step 4 (Marginal constraint at $t=1$).} When $t\mapsto\mu_t$ is left-continuous at $t=1$, i.e. $\mu_{1^-}= \mu_{1}$, we can modify curves in $D(I:X)$ where $\tilde{\pi}$ is concentrated to let them be left-continuous at $t=1$. After that, both sides in \eqref{eq:pfThm3.3} depend continuously on $t$ around $t=1$, leading to $(e_{1})_{\#}\tilde{\pi}=\mu_1$.
\\
When the left limit is $\mu_{1^-}\neq \mu_{1}$, take $\sigma\in\opt(\mu_{1^-},\mu_1)$ and let us define the measure $\hat{\pi}\coloneqq C_{\#}\tilde{\pi}$ by performing the continuation at $t=1$, where 
\begin{equation}
  C\colon D(I:X)\to D(I:X),\quad C(u)(t)\coloneqq \left\{\begin{array}{ll}
        u(t),& t\in[0,1)\\
        \lim_{s\nearrow 1}u(s),& t=1
    \end{array}\right..
\end{equation}
Denote by $D_1(I:X)\subset D(I:X)$ the closed subset of all Càdlàg curves left-continuous at $t=1$.
Notice that there is a natural Borel isomorphism between $D_1(I:X)$ and $D([0,1):X)$. 
Therefore, we may regard $D(I:X)$ as the product of Polish spaces $D([0,1):X)$ and $X$.
In this way, $\hat\pi\in P(D([0,1):X)\times X)$ with $\mathrm{Pr}^2_\#\hat\pi=\mu_{1^-}$, and thus it can be glued together with $\sigma$ (along $X$) to obtain a probability measure $\omega$ on $D([0,1):X)\times X\times X$ by the gluing lemma, see e.g. \cite[Lemma 3.1]{Ambrosio2013UserGuide}. 
The projection $\mathrm{Pr}^{1,3}_{\#}\omega$ of $\omega$ on the first and third marginal will be the desired $\tilde{\pi}$ (with a slight abuse of notation, still denoted by $\tilde{\pi}$).
\\
Clearly, the new $\tilde{\pi}$ verifies \ref{itm:marginalprop} and \eqref{eq:Dmu}.
Additionally, from the construction that $\mathrm{Pr}^{2,3}\omega=\sigma\in \opt(\mu_{1^{-}},\mu_1)$, we have
\begin{equation}
    W_1(\mu_{1^{-}},\mu_1)=\int d(\gamma_{1^{-}},\gamma_1)\d\tilde{\pi}(\gamma).
\end{equation}
Together with \cref{lemma:normalizedBV}, it means 
\begin{equation}
    \var(\mu;I)=\int\var(\gamma;I)\d\tilde{\pi}(\gamma).
\end{equation}
\\
\textbf{Proof of the final claim.}
    Suppose first that $(\mu_t)$ is absolutely continuous, then in particular, it is a curve of bounded variation and thus \ref{itm:support}-\ref{itm:totalvarian} already hold. 
    To derive \eqref{eq:dotmu_BV}, observe that for $\mathcal{L}^1$-a.e. $t\in I$
    
    \begin{align}
        |\dot\mu|(t)
        & = \lim_{h \rightarrow 0 } \frac{1}{|h|}\int_{t}^{t+h} |\dot\mu|(s) \d s \\
        & = \lim_{h \rightarrow 0 } \frac{1}{|h|} |D\mu|([t,t+h]) \\
        & = \lim_{h \rightarrow 0 } \int \frac{  |D\gamma|([t,t+h])}{|h|} \d\pi(\gamma) \\
        & \geq \lim_{h \rightarrow 0 } \int \frac{  d (\gamma_{t}, \gamma_{t+h})}{|h|} \d\pi(\gamma) \\ & \geq \lim_{h \rightarrow 0 } \frac{W_1(\mu_t,\mu_{t+h})}{|h|} \eqqcolon |\dot\mu_t| = |\dot\mu|(t)
    \end{align}
    where in the third step, we used \eqref{eq:Dmu}. Therefore, the statement follows whenever $(\mu_t)$ is absolutely continuous.
    \\
    We now claim that the argument above works also in the non-absolutely continuous case.
    Indeed, \cref{lemma:metric_derivative} guarantees the equivalence between $|\dot\mu|(t)$ and the second line as well as the validity of the last equality (recall the notation explained in \cref{subsec:BV-curves}). In the other steps, absolute continuity was not used. Hence, the conclusion follows.
 \end{proof}

\subsection{Remarks on the main results}\label{subsec:remarks_results}
In this section, we shed light on the main results by providing some examples.

First of all, note that in general, we can not expect any uniqueness of $\tilde{\pi}$ in \cref{thm:lift_BV}. Indeed, the uniqueness is not true even in Lisini's original result \cite[Theorem 5]{Lisini} for $p>1$. However, when $p>1$, there are cases where the lift is unique. For instance, when the underlying space is non-branching, then the lift of any constant-speed geodesic must be unique (see e.g. \cite[Proposition 3.16]{Ambrosio2013UserGuide}).
The following example illustrates that in the case $p=1$, where the cost lacks strict convexity, uniqueness fails throughout, even in Euclidean spaces.

\begin{example}[Nonuniqueness of lifts]\label{exp:nonuniquelift} Let $\mu_0$ and $\mu_1$ be two probability measures on $\R$ supported inside $[-2,-1]$ and $[1,2]$, respectively. Clearly $t\mapsto\mu_t\coloneqq(1-t)\mu_0+t\mu_1$ is a constant-speed geodesic under $W_1$-distance and notice that every coupling between $\mu_0$ and $\mu_1$ is optimal. For any coupling induced by a Borel map $T$, let us define a family of curves labelled by $\alpha \in [0,1]$ and $x \in \text{supp}(\mu_0)$ in the following way
\begin{equation}
t \mapsto \gamma_t^{(\alpha,x)} \coloneqq x\mathds{1}_{[0,\alpha)}(t)+T(x)\mathds{1}_{[\alpha,1]}(t). 
\end{equation}
This is in fact a generalized version of the construction in \cref{ex:theBeginning} (in which we had $\text{supp}(\mu_0) =\{ 0\}$ and $T(x) = 1$).
Now, similar to that example, one can check that the measure $$\pi\coloneqq (\gamma^{(\cdot,\cdot)})_{\#}(\mathcal{L} |_{[0,1]} \otimes \mu_0 )$$ satisfies \ref{itm:support}-\ref{itm:totalvarian} in the theorem above.
Whenever there exist at least two transport maps (e.g. in the case that $\mu_0$ and $\mu_1$ are uniformly distributed) then the lift $\pi\in P (D(I:X))$ of $(\mu_t)$ will no longer be unique.  
\end{example}

Another natural question is to what extend, an AC-curve in the $1$-Wasserstein space has lifts on AC-curves satisfying the optimality condition \eqref{eq:Dmu}.
Recall that \cref{ex:theBeginning} already provides an extreme case where no lift on continuous curves is possible.
The example below demonstrates that the existence of lifts on AC-curves is not a guarantee for finding an optimal one.

\begin{example}[Non-optimality of AC lifts]\label{exp:AC_not_enough}
    Take $\gamma \in\mathcal{AC}(I:\R^2)$ with unit length and assume $\gamma$ is not length-minimizing, i.e., $d(\gamma_0,\gamma_1)<1$. 
    Consider $(\mu_t) \subset P_1(\R^2)$ defined as 
    \begin{equation}
        \mu_t \coloneqq \frac{1}{2} \Big( (1-t) \delta_{\gamma_0} + t \delta_{\gamma_1}\Big) + \frac{1}{2} \mathcal{H}^1|_{\gamma}, \quad t \in  [0,1],
    \end{equation}
    where $\mathcal{H}^1|_{\gamma}$ is the 1-dimensional Hausdorff measure of $\gamma$.
    First of all, it can be readily checked that $(\mu_t)$ is a constant-speed geodesic.
    Secondly, there exists a lift $\pi$ of $(\mu_t)$ concentrated on AC-curves. For example, take two families of AC-curves
    \begin{equation}
        \gamma^{(1,\alpha)}(t)\coloneqq\left\{\begin{array}{ll}
           \gamma(0),  & t\leq\alpha \\
            \gamma(t-\alpha),& t>\alpha
        \end{array}\right.,\quad\gamma^{(2,\alpha)}(t)\coloneqq\left\{\begin{array}{ll}
           \gamma(t+\alpha),  & t\leq1-\alpha \\
            \gamma(1),& t>1-\alpha
        \end{array}\right..
    \end{equation}
    Then $\pi\coloneqq \frac12\gamma^{(1,\cdot)}_{\#}\L^1+\frac12\gamma^{(2,\cdot)}_{\#}\L^1$ can be checked to satisfy $(e_t)_{\#}\pi=\mu_t$ for each $t\in [0,1]$.
    However, $|\dot\mu_t| <  \int |\dot\gamma_t| \d\pi(\gamma)$.
    Actually, there is no way for any lift $\pi$ concentrated on AC-curves to achieve the equality \eqref{eq:Dmu}. 
    Because if $(\mu_t)$ is optimally transported along continuous curves, they have to be length-minimizing. 
    But on the other hand, (almost all) curves in $\pi$ have to lie inside $\gamma$, as each $\mathrm{supp}(\mu_t)$ does.
\end{example}

Observe that if $\gamma$ in the example above is length-minimizing, then the constructed lift is indeed optimal. In this case, all measures $\mu_t$ live in a convex set. 
One can ask whether the strategy above could yield an optimal lift concentrated on AC-curves if we restrict all $\mu_t$ to be fully supported on a common convex domain. 
We further add the assumption $\mu_t=\rho_t\L^n$ of absolute continuity of measures, trying to exclude the teleporting phenomenon, which appears for instance when replacing the one-dimensional Hausdorff measure by a higher-dimensional one in \cref{exp:AC_not_enough}.
However, such convincing-sounded assumptions (even with uniform bounds on densities $\rho_t$) turn out to fail again for obtaining a lift on AC-curves, at least in higher dimensions, as shown below. 
This is again to emphasize that it is necessary to relax the classical notion of lift and consider a larger class of curves.

\begin{example}\label{exp:2D_example}
    Consider a curve of probability measures on $\R^2$ defined as
    \begin{align} \label{eq:mu_2D_example}
        \mu_t\coloneqq \frac12\rho_t\L^2+\frac12\L^2\lvert_{[0,1]\times[0,1]}, \quad t \in  [0,1],
    \end{align}
    where the density $\rho_t\colon \R^2\to \R$ at time $t$ is defined as 
    \begin{align}
        \rho_t\coloneqq \left\{\begin{array}{ll}
            \frac1\varepsilon, & \quad\mathrm{if\ } x\in [0,\varepsilon]\times[t,1]\bigcup[1-\varepsilon,1]\times[0,t],  \\
            0 & \quad\mathrm{otherwise}. 
        \end{array}\right.
    \end{align}
     Here $\frac12>\varepsilon>0$ is a fixed parameter. Measure $\mu_t$ is shown in \cref{fig:2D_example} (left).
     Since the curve $t\mapsto \rho_t\L^2$ is a constant speed 1-Wasserstein geodesic, so is the curve $(\mu_t)$. 
    While the curve $(\mu_t)$ has infinitely many lifts, all of them have the property that they (up to neglecting a zero measure set of curves) only transport horizontally. This allows us to reduce the study of a possible lift to the 1-dimensional problem of the following measures on $[0,1]$. 
    Given any $y\in [0,1]$ (corresponding to the slices of the measures $\mu_t$ at height $y$), we define a curve of probability measures on $\R$ as
    \begin{align}\label{eq:mu-y_2D_example}
        \mu_t^y\coloneqq \begin{cases} \frac1{2\varepsilon}\L^1\lvert_{[0,\varepsilon]}+\frac12\L^1\lvert_{[0,1]},&\quad\mathrm{if\ } t\le y ;\\
        \frac1{2\varepsilon}\L^1\lvert_{[1-\varepsilon,1]}+\frac12\L^1\lvert_{[0,1]},&\quad\mathrm{if\ } t> y;
        \end{cases}
    \end{align}
    \\
    as shown in \cref{fig:2D_example} (right). Consider now any lift $\pi^y$ of $(\mu^y_t)$ as in the \cref{thm:lift_BV}.
    Since $|D\mu^y|=|D\mu^y|^J=\frac{1-\varepsilon}2\delta_y$, we have that $\pi^y(\Gamma_y)=\frac{1-\varepsilon}2>0$, where $\Gamma^y$ is the collection of curves that jump at $t=y$.
    We conclude that any optimal lift $\pi$ of $(\mu_t)$ as in \cref{thm:lift_BV} gives positive mass for the set of non-absolutely continuous curves.
    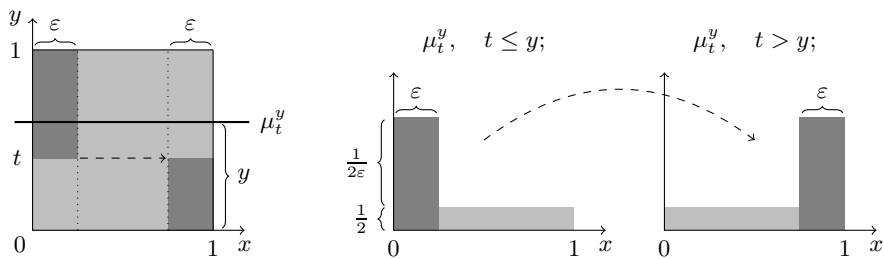
\begin{figure}[h]
    \begin{tikzpicture}[scale=2.4/4]
    \path [fill=lightgray] (0,0) rectangle (4,4);
    \path [fill=gray] (0,2/5*4) rectangle (1,4);
    \path [fill=gray] (3,0) rectangle (4,2/5*4);
    \draw [<->] (0,4.7) node [left] {\scriptsize $y$} -- (0,0) -- (4.7,0) node [below] {\scriptsize $x$};
    \draw (0.1,0.1) node[anchor=north east] {\scriptsize $0$};
    \draw (0,2/5*4) node[anchor=east] {\scriptsize $t$};
    \draw (0,4) node[anchor=east] {\scriptsize $1$};
    \draw (4,0) node[anchor=north] {\scriptsize $1$};
    \draw [decorate, decoration = {brace}] (0+0.05,4+0.1) --  (1-0.05,4+0.1);
    \draw (0.5,4+0.2) node[anchor=south] {\scriptsize $\varepsilon$};
    \draw [decorate, decoration = {brace}] (3+0.05,4+0.1) --  (4-0.05,4+0.1);
    \draw (3.5,4+0.2) node[anchor=south] {\scriptsize $\varepsilon$};
    \draw (4,0) -- (4,4);
    \draw (4,4) -- (0,4);
    \draw [dashed][->] (1+0.05,2/5*4) -- (3-0.05,2/5*4);
    \draw [dotted] (1,0) -- (1,4);
    \draw [dotted] (3,0) -- (3,4);
    \draw [decorate, decoration = {brace,mirror}] (4.2,0+0.05) --  (4.2,4*3/5-0.05);
    \draw (4.3,4*3/5/2) node[anchor=west] {\scriptsize $y$};
    \draw [thick] (-0.4,3/5*4) -- (4.8,3/5*4);
    \draw (4.8,3/5*4) node[anchor=west] {\scriptsize $\mu_t^y$};
    \path [fill=lightgray] (8,0) rectangle (8+4,0.5);
    \path [fill=gray] (8,0) rectangle (8+1,2.5);
    \path [fill=lightgray] (6+8,0) rectangle (6+8+4,0.5);
    \path [fill=gray] (6+3+8,0) rectangle (6+3+8+1,2.5);
    \draw [<->] (8,3.5) -- (8,0) -- (8+4.7,0) node [below] {\scriptsize $x$};
    \draw (8,0) node[anchor=north] {\scriptsize $0$};
    \draw (8+4,0) node[anchor=north] {\scriptsize $1$};
    \draw [decorate, decoration = {brace}] (8+0.05,2.5+0.1) --  (8+1-0.05,2.5+0.1);
    \draw (8+0.5,2.5+0.2) node[anchor=south] {\scriptsize $\varepsilon$};
    \draw [decorate, decoration = {brace}] (8-0.2,0) --  (8-0.2,0.5);
    \draw (8-0.3,0.5/2) node[anchor=east] {\scriptsize $\frac{1}{2}$};
    \draw [decorate, decoration = {brace}] (8-0.2,0.5+0.05) --  (8-0.2,2.5-0.05);
    \draw (8-0.3,3/2) node[anchor=east] {\scriptsize $\frac{1}{2\varepsilon}$};
    \draw (8+2,3.7) node[anchor=south] {\scriptsize $\mu_t^y, \quad t \leq y;$};
    \draw [<->] (6+8,3.5) -- (6+8,0) -- (6+8+4.7,0) node [below] {\scriptsize $x$};
    \draw (8+6,0) node[anchor=north] {\scriptsize $0$};
    \draw (8+4+6,0) node[anchor=north] {\scriptsize $1$};
    \draw [decorate, decoration = {brace}] (8+6+3+0.05,2.5+0.1) --  (8+6+3+1-0.05,2.5+0.1);
    \draw (8+6+3+0.5,2.5+0.2) node[anchor=south] {\scriptsize $\varepsilon$};
    \draw (6+8+2,3.7) node[anchor=south] {\scriptsize $\mu_t^y, \quad t > y;$};
    \draw [dashed][->] (8+2,2) .. controls (8+2+2,3.5) and (8+2+6-2,3.5) .. (8+2+6,2);
\end{tikzpicture}
\captionsetup{font=scriptsize}
\caption{(\cref{exp:2D_example}) \textbf{Left} image shows the measure $\mu_t$ in \eqref{eq:mu_2D_example} for $\varepsilon =1/4$ at time $t=2/5$. \textbf{Right} image shows the conditional measure $\mu_t^y$ in \eqref{eq:mu-y_2D_example} given $y=3/5$ at different times.}
\label{fig:2D_example}
\end{figure}
\end{example}

\section{Applications}\label{sec:applications}
\subsection{Characterization of BV-curves}

An immediate consequence of combining \cref{thm:BVSk} and \cref{thm:lift_BV} is that one can characterize BV-curves in 1-Wasserstein spaces:

\begin{corollary}[Characterization of BV-curves in 1-Wasserstein spaces]\label{cor:chrtz_BV}
    Let $(\mu_t)\subset P(X)$ with $\mu_0 \in P_1(X)$.
    Then
        $(\mu_t) \in \mathcal{BV}(I:P_1(X))$ if and only if there exists $\pi\in P(D(I:X))$ so that
        \begin{enumerate}[label=(\roman*),font=\normalfont,leftmargin=30pt]
            \item $\pi$ is concentrated on $\mathcal{BV}(I:X)\subset D(I:X)$;
            \item $(e_t)_\#\pi=\mu_t$ for all $t\in [0,1]$; 
            \item \label{itm:san} 
            $
            \int |D \gamma| (I) \d \pi < \infty.
            $
            \end{enumerate}
\end{corollary}
Characterization of AC-curves in 1-Wasserstein space \emph{using their lifts} however remains challenging. 
A naive extension of the characterization in the case $p>1$ to the case $p=1$ would lead to $\int \int_{I} |\dot\gamma_t| \d t \d \pi < \infty$, which is a well-defined condition as the metric derivative for BV-curves still exists $\mathcal{L}^1$-a.e. 
However, this condition does not guarantee even continuity, let alone absolute continuity.
In fact, it is already weaker than \ref{itm:san}.
Recall \cref{ex:theBeginning}, where \emph{any} lift of the absolutely continuous curve $(\mu_t)$ is concentrated on discontinuous curves.
Continuous curves of bounded variation, on the other hand, can be easily characterized, for which the following observation is useful: 

\begin{proposition}\label{prop:jump}
Under the assumptions of \cref{thm:lift_BV}, we have for all $t\in I$
\begin{equation}\label{eq:Dmu_J}
    |D\mu|^{J} ( \{t\} ) = \int |D\gamma|^{J} ( \{t\} ) \d\pi (\gamma).
\end{equation}
\end{proposition}
\begin{proof}
    By considering the atomic part in the Lebesgue decomposition of the variation measure and using equation \eqref{eq:Dmu}, we obtain
    \begin{equation}
            |D\mu|^{J} ( \{t\} )
        = |D\mu| ( \{t\} )  =  \int |D\gamma| ( \{t\} ) \d\pi (\gamma) = \int |D\gamma|^{J} ( \{t\} ) \d\pi. 
    \end{equation}
\end{proof}
Equation \eqref{eq:Dmu_J} simply means that the jump size in the 1-Wasserstein space is obtained by taking the average over all jumps in the underlying space.
An implication is that if $\mu_t$ jumps at time $t$ (i.e. the left-hand side of \eqref{eq:Dmu_J} is non-zero), then at least some curves must jump at this time as well.
Notice that this does not contradict the observation in \cref{ex:theBeginning}. 
Even though all the underlying curves in the example jump, this does not lead to a jump in $(\mu_t)$ since at any time $t$, only one curve jumps and thus has measure zero in the lift.
As a result, we conclude that a BV-curve $(\mu_t)$ is continuous if and only if for all $t\in I$, the set of curves $(\gamma_t)$ which has jump at $t$ has measure zero.
This is formally stated in the corollary below:

\begin{corollary}[Characterization of continuous BV-curves in 1-Wasserstein spaces]\label{cor:chrtz_CBV}
    We have $(\mu_t) \in \mathcal{BV}(I:P_1(X)) \cap C(I:P_1(X))$ if and only if there exists $\pi\in P(D(I:X))$ that in addition to (i)-(iii) of \cref{cor:chrtz_BV}, satisfies
    \begin{enumerate}[label=(\roman*),font=\normalfont,leftmargin=30pt]\setcounter{enumi}{3}
    \item for all $t\in I$
    $$\int|D\gamma|(\{t\})\d\pi(\gamma)=0.$$
    \end{enumerate}
\end{corollary}

\subsection{Characterization of geodesics}
Here we apply the main theorem to study geodesics in 1-Wasserstein spaces. To this end, we shall consider a relaxed notion of geodesic:
\begin{definition}[$\mathcal{BV}$-geodesics]\label{def:BV-geodesics}
We call $\gamma \in D([0,1]:X)$ a $\mathcal{BV}$-geodesic with respect to distance $d$ if $d(\gamma_0,\gamma_1)=|D \gamma|([0,1])$.
\end{definition}

\begin{theorem}[Characterization of geodesics in 1-Wasserstein spaces]\label{thm:WassGeo}
Let $(\mu_t) \subset P(X)$ with $\mu_0 \in P_1(X)$. Then $(\mu_t)$ is
\begin{itemize}[leftmargin=*]
    \item \textbf{$\mathcal{BV}$-geodesic} with respect to $W_1$ distance if and only if there exists $\pi\in P(D(I:X))$ so that 
    \begin{enumerate}[label=(\roman*),font=\normalfont,leftmargin=30pt]
        \item \label{itm:concgeo} $\pi$ is concentrated on the set of $\mathcal{BV}$-geodesics;
        \item \label{itm:liftgeo}$(e_t)_\#\pi=\mu_t$ for all $t \in I$;
        \item \label{itm:koo} $W_1(\mu_0,\mu_1)=\int d(\gamma_0,\gamma_1)\d\pi(\gamma)< \infty$.
    \end{enumerate}
    \item continuous and \textbf{length minimising} if and only if in addition to \ref{itm:concgeo}-\ref{itm:koo}, $\pi$ satisfies
    \begin{enumerate}[label=(\roman*),font=\normalfont,leftmargin=30pt]\setcounter{enumi}{3}
    \item for all $t\in I$ $$\int|D\gamma|(\{t\})\d\pi(\gamma)=0.$$
    \end{enumerate}
    \item \textbf{constant-speed geodesic} if and only if in addition to (i)-(iii), $\pi$ satisfies
    \begin{enumerate}[label=(\roman*),font=\normalfont,leftmargin=30pt]\setcounter{enumi}{4}
    \item\label{itm:csg} for $\mathcal{L}^1$-a.e. $t\in I$ $$\lim_{h \rightarrow 0 } \int \frac{  d (\gamma_{t}, \gamma_{t+h})}{|h|} \d\pi(\gamma) = W_1(\mu_0, \mu_1).$$ 
    \end{enumerate}
\end{itemize}
\end{theorem}
\begin{proof}
    Suppose first that $\pi$ satisfies \ref{itm:concgeo}--\ref{itm:koo}.
    Then by Theorem \ref{thm:BVSk} we have that
    \begin{align}
        W_1(\mu_0,\mu_1)\le |D\mu|([0,1])\le \int |D\gamma|([0,1]) \d\pi(\gamma)=\int d(\gamma_0,\gamma_1)\d\pi(\gamma)=W_1(\mu_0,\mu_1).
    \end{align}
    Hence, all the above inequalities are actually equalities, and thus $(\mu_t)$ is $\mathcal{BV}$-geodesic.
    
    Suppose now that $(\mu_t)$ is a $\mathcal{BV}$-geodesic, and let $\pi$ be given by Theorem \ref{thm:lift_BV}. Then \ref{itm:liftgeo} holds and we have 
    \begin{align}
        W_1(\mu_0,\mu_1)=|D\mu|([0,1])=\int |D\gamma|([0,1])\d\pi(\gamma)\ge \int d(\gamma_0,\gamma_1)\d\pi(\gamma)\ge W_1(\mu_0,\mu_1),
    \end{align}
    which proves \ref{itm:koo}. Furthermore, since the first inequality is due to the pointwise inequality $|D\gamma|([0,1])\ge d(\gamma_0,\gamma_1)$, we can actually conclude that there has to be pointwise equality for $\pi$-almost every curve $\gamma$, and hence $\pi$ is concentrated on $\mathcal{BV}$-geodesics.
    
    The claim about continuous and length minimizing follows immediately after considering \cref{prop:jump}.
    As for the characterization of constant-speed geodesics, thanks to the equality \eqref{eq:dotmu_BV}, it is enough to show the sufficiency, i.e. to show \ref{itm:concgeo}-\ref{itm:koo} and \ref{itm:csg} imply being constant-speed geodesic.
    By \cref{thm:lift_BV} and \ref{itm:koo}, we always have
    \begin{align}
    \int\lim_{h \rightarrow 0 } \int \frac{  d (\gamma_{t}, \gamma_{t+h})}{|h|} \d\pi(\gamma)\d t
    &=\int |\dot{\mu}|(t)\d t \\
    &\leq|D\mu|([0,1]) =\int |D\gamma|([0,1])\d\pi (\gamma)\\
    &=\int d(\gamma_0,\gamma_1)\d\pi(\gamma)=W_1(\mu_0,\mu_1).
    \end{align}
    Hence, once \ref{itm:csg} holds, equality holds everywhere in the above and in particular, 
    \[
    |D\mu|=W_{1}(\mu_0,\mu_1)\L^1|_{[0,1]}
    \]
    which means that $(\mu_t)$ has to be a constant-speed geodesic.
\end{proof}

\subsection{Regularity of the curves in superposition}
\cref{ex:theBeginning} illustrates that the superposition of a family of discontinuous curves can produce an absolutely continuous Wasserstein curve.
One could naturally ask similar questions, e.g., is it possible to get an absolutely continuous curve from superposing continuous singular curves? 
Here we answer these questions by investigating three different scenarios, where the lift as in \cref{thm:lift_BV} is concentrated \emph{purely} on either absolutely continuous (AC), continuous singular (CS), or jump (J) curves.
The answer is summarized in \cref{table:gamma_mu} and elaborated upon thereafter.

{\scriptsize
\begin{table}[h]
    \captionsetup{font=scriptsize}
    \caption{Curves $(\mu_t)$ in 1-Wasserstein spaces vs. curves $\gamma$ in their lifts as in \cref{thm:lift_BV}. The lift is assumed to be concentrated purely on either absolutely continuous (AC), continuous singular (CS), or jump (J) curves. We specify whether superposition is possible ($\checkmark$) or not ($\times$).}
    \begin{tabular}{ c  c | l l l l l }\label{table:gamma_mu}
     & \hspace{-10pt} $(\mu_t)$ & \multicolumn{1}{c}{AC} & & \multicolumn{1}{c}{CS} & & \multicolumn{1}{c}{J} \\
     $\gamma$  & & & & & &\\
    \hline
    & & & &\\
    AC &  &   $\checkmark$ (Take $\mu_t = \delta_{\gamma_t}$) & & $\times$ (\cref{rem:AC_implies_AC})  & & $\times$ (\cref{rem:AC_implies_AC}) \\
    & & & &\\
    CS &  &   $\checkmark$ (\cref{exp:periodic_extension_2}) & & $\checkmark$ (Take $\mu_t = \delta_{\gamma_t}$) & & $\times$ (\cref{prop:jump}) \\
    & & & & & &\\
    J &  &  $\checkmark$ (\cref{ex:theBeginning}) & & $\checkmark$ (Equation \eqref{eq:mu_CS}) & & $\checkmark$ (Take $\mu_t = \delta_{\gamma_t}$) \\
    & & & & & &
    \end{tabular}
\end{table}
}

First, note that we can always have a curve with the same regularity as the underlying curves by simply taking $\mu_t = \delta_{\gamma_t}, t\in I $ (diagonal entries of the table).
Secondly, if all curves $\gamma$ are AC, then $(\mu_t)$ is necessarily AC and thus cannot be CS or J due to the following simple observation: 
\begin{rem}[A sufficient condition for absolutely continuity]\label{rem:AC_implies_AC}
  Under the assumptions of \cref{thm:lift_BV}, if the lift $\tilde{\pi}$ is merely concentrated on $\mathcal{AC}^1(I:X)$, then $(\mu_t) \in \mathcal{AC}^1(I:P_1(X))$.
  This directly follows from \eqref{eq:Dmu},
  \begin{align}
    |D\mu|=\int|D\gamma|\d\tilde{\pi}(\gamma) = \int \int_I |\dot\gamma_t| \d t \d\tilde{\pi}(\gamma) = \int_I \left( \int |\dot\gamma_t|  \d\tilde{\pi}(\gamma) \right) \d t,
    \end{align}
    which implies that $|D\mu| \ll \mathcal{L}^1|_{I}$ and thus $(\mu_t)$ is an AC-curve.
\end{rem}
Next, the case $\gamma$-CS, $(\mu_t)$-J is not possible. As explained after \cref{prop:jump}, if $(\mu_t)$ jumps, then at least some curves must jump as well. The opposite case, $\gamma$-J, $(\mu_t)$-CS, can however happen. Take, for example,
\begin{align}\label{eq:mu_CS}
    \mu_t \coloneqq (1-c(t))\delta_0+ c(t)\delta_1, \quad t \in I,
\end{align}
where $c: [0,1] \to [0,1]$ is the Cantor function. Then, in the same way as \cref{ex:theBeginning}, one can construct a lift that is concentrated on jump curves. Lastly, in \cref{exp:periodic_extension_2} below, we show that the case $\gamma$-CS, $(\mu_t)$-AC is possible as well.

As a final remark, it is interesting to notice that all cases in the upper diagonal of \cref{table:gamma_mu} turn out to be not possible.
One can conclude that this is simply due to the fact that one can produce regular Wasserstein curves out of irregular curves, while producing irregular Wasserstein curves purely out of regular curves is not possible. 

\begin{example}[Constant-speed $W_1$-geodesics out of arbitrary BV-curves]\label{exp:periodic_extension_2}
Here we provide a fairly general way of constructing constant-speed $W_1$-geodesics (which are in particular absolutely continuous) for which the lift is concentrated on an arbitrary set of BV-curves, in particular, on a set of non-absolutely continuous curves.
Consider $X = \mathbb{R}$ with the Euclidean distance and let $\sigma_0\in P([0,1])$ be an arbitrary probability measure with $\sigma_0(\{1\})=0$. 
The aim of this example is to construct a lift $\pi$ as in \cref{thm:lift_BV} so that the corresponding Wasserstein curve is a constant-speed geodesic connecting $\delta_0$ to $\delta_1$, and $\pi$ is concentrated on curves whose total variation measure are translates of $\sigma_0$.
Then by arbitrarily choosing $\sigma_0$, we will be able to obtain different geodesics. 
\\
First, we define a measure $\sigma$ by periodically extending $\sigma|_{[n,n+1)}=\sigma_0|_{[0,1)}$ for all $n\in\Z$. Define a family of curves $\{ \gamma^{(\alpha)} \}_\alpha$ labelled by $\alpha \in [0,1]$
\[
\gamma^{(\alpha)}(t)\coloneqq \sigma((\alpha,\alpha+t]), \quad t \in I,
\]
which forms a Borel map $\gamma^{(\cdot)}: \alpha \mapsto \gamma^{(\alpha)}$ from $I$ to $D(I:[0,1])$. Indeed, by \cref{lma:Boreleval}, it follows from the fact that for every $t$, the composition
\[
[0,1]\overset{\gamma^{(\cdot)}}{\to} D(I:[0,1])\overset{e_t}\to [0,1] 
\]
is Borel. Notice that the map $\alpha \mapsto \sigma((\alpha,\alpha +t])$ is the difference of two increasing functions, namely, $\alpha\mapsto\sigma((0,\alpha])$ and $\alpha\mapsto\sigma((0,\alpha+t])$.
\\
Define $\pi\in P(D(I:X))$ as 
\begin{align}
    \pi\coloneqq (\gamma^{(\cdot)})_{\#}\L^1,
\end{align}
and $\mu_t\coloneqq (e_t)_\#\pi$.
Clearly $\pi$ is concentrated on $\mathcal{BV}(I:[0,1])$ so by \cref{thm:BVSk}, $(\mu_t) \in \mathcal{BV} (I: P_1(X))$ and
 \begin{align}\label{eq:geodesic_from_cantor}
     |D\mu|\le \int |D\gamma|\d\pi(\gamma)=\int_0^1|D\gamma^{(\alpha)}| \d\L^1( \alpha ).
 \end{align}
As the family $\{\gamma^{(\alpha)}\}_\alpha$ is given by translation, we have
\begin{align}\label{eq:sumDgamma}
    \int_0^1|D\gamma^{(\alpha)}|\d \L^1( \alpha)=\L^1.
\end{align}
Actually, for any $f\in C([0,1]:\R)$,
 \begin{align}
     \iint f(t) |D\gamma^{(\alpha)}|(\d t)\L^1(\d \alpha)&=\iint f(\iota(t+\alpha))|D\gamma^0|(\d t)\d \alpha
     \\&=\iint f(\iota(t+\alpha))\d\sigma(t)\d \alpha
     \\ &=\iint f(\alpha)\d \alpha\d\sigma(t)=\int f\d\L^1,
 \end{align}
where $\iota:\R\to [0,1)\simeq\R/\Z$ is the quotient map.
\\
Notice for all $\gamma^{(\alpha)}$, we have $\gamma^{(\alpha)}(0)=0$ and $\gamma^{(\alpha)}(1)=1$. So $W_1(\mu_0,\mu_1)=W_1(\delta_0, \delta_1)=1$. On the other hand, combining \eqref{eq:geodesic_from_cantor} and \eqref{eq:sumDgamma}, we have $|D\mu|=\L^1$. In other words, $(\mu_t)$ is a constant-speed geodesic. 
\\
In conclusion, we have constructed a constant-speed geodesic $(\mu_t)$ and a lift $\pi$ of $(\mu_t)$ concentrated on BV-curves whose variation measures are cyclical translations of $\sigma_0$. Now, different choices $\sigma_0$ give rise to different $W_1$-geodesics between $\delta_0$ and $\delta_1$, as shown in \cref{fig:constant-speed-geodesics}. In particular,
\begin{itemize}
    \item $\sigma_0= \mathcal{L}|_{[0,1]}$ corresponds to the trivial geodesic $ \mu_t = \delta_t$ (\cref{fig:constant-speed-geodesics} (A)), which is also the (unique) constant-speed geodesic for $p>1$. 
    \item $\sigma_0= \delta_0$ corresponds to geodesic $\mu_t = (1-t)\delta_0+t\delta_1$, studied in \cref{ex:theBeginning}. 
    \item Finally, by choosing $\sigma_0$ to be a probability measure with no atoms and no absolutely continuous part, we get that $\pi$ is concentrated on BV-curves that are continuous but not absolutely continuous.
\end{itemize}

\begin{figure}[h]
  \centering
    \begin{subfigure}[b]{0.36\linewidth}
        \captionsetup{font=scriptsize, position=top}
        \subcaption{$\,$}
        \centering
        \includegraphics[width=\linewidth]{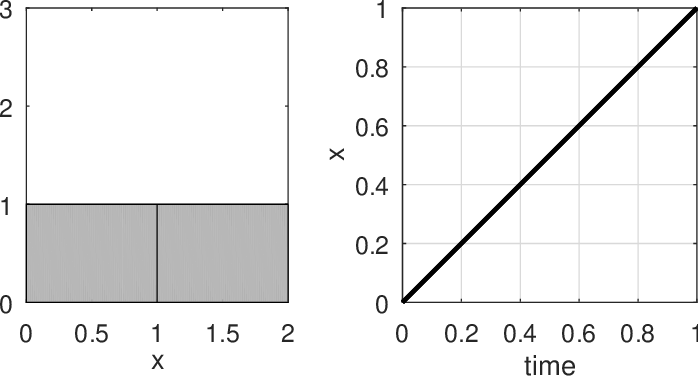}
  \end{subfigure}
  \hspace{0.05\linewidth}
  \begin{subfigure}[b]{0.36\linewidth}
        \captionsetup{font=scriptsize, position=top}
        \subcaption{$\,$}
        \centering
        \includegraphics[width=\linewidth]{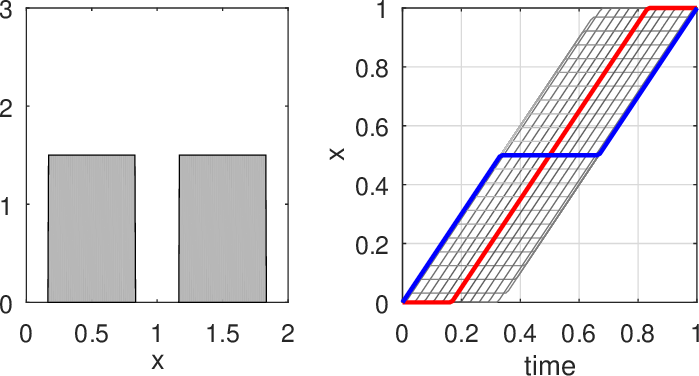}
  \end{subfigure}
  \\
  \begin{subfigure}[b]{0.36\linewidth}
        \captionsetup{font=scriptsize, position=top}
        \subcaption{$\,$}
        \centering
        \includegraphics[width=\linewidth]{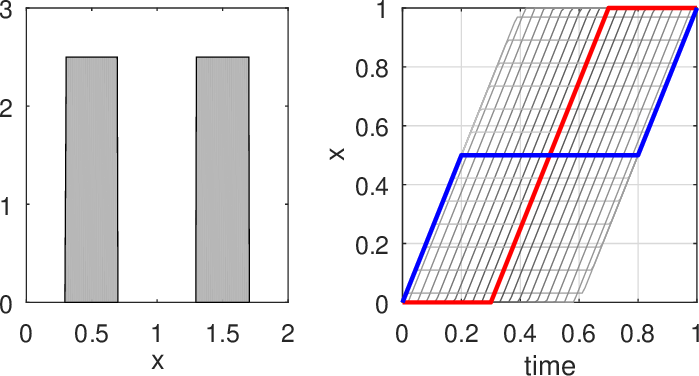}
 \end{subfigure}
 \hspace{0.05\linewidth}
  \begin{subfigure}[b]{0.36\linewidth}
        \captionsetup{font=scriptsize, position=top}
        \subcaption{$\,$}
        \centering
        \includegraphics[width=\linewidth]{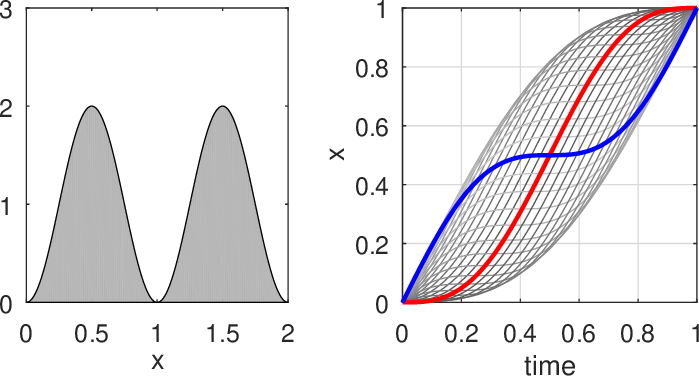}
  \end{subfigure}
 \captionsetup{font=scriptsize}
\caption{(\cref{exp:periodic_extension_2}) Construction of different constant-speed geodesics in $P_1(\mathbb{R})$ between $\delta_0$ and $\delta_1$.
\textbf{Left} sub-figures show example measures $\sigma_0$ and their periodic extension $\sigma$.
\textbf{Right} sub-figures show sample curves $\gamma^{(\alpha)}, \alpha \in [0,1]$. Curves with $\alpha=0, 0.5$ are highlighted in red and blue, respectively.
}
\label{fig:constant-speed-geodesics}
\end{figure}

\end{example}

\subsection{Continuity equation in discrete setting}
\cref{thm:lift_BV} is a useful tool for the study of BV-curves in 1-Wasserstein spaces.
In the continuous setting, it is well known that whenever the space $X$ has a
kind of differential structure, absolutely continuous curves $(\mu_t) \subset P_p(X)$, for $p>1$, are related to solutions of the continuity equation (see e.g. \cite[Chapter 8]{AGSmetricflow}). 
More precisely, one can find a  time-dependent Borel
velocity field $v_t: X \to X$ of $(\mu_t)$ so that the continuity equation
\begin{align}
    \partial_t \mu_t+ \nabla \cdot (v_t\mu_t)=0 \quad \mathrm{in} \, X \times I
\end{align}
holds and $|\dot\mu_t|^p=\int |v_t|^p(x)\d \mu_t(x)$.

Concerning the continuity equation, the case $p=1$ is far more involved, not least due to the presence of non-localities as seen already in \cref{ex:theBeginning}. 
While the exponent $p=1$ creates great difficulties in the continuous setting, it also opens up the possibility of studying analogous questions in the discrete setting.
The discrete counterpart to the continuity equation, sometimes referred to as the current equation, is also studied in the literature and has a tight connection with Markov chains.
In \cite{Leonard2016Lazy}, among other things, L\'eonard derives a Benamou--Brenier type formula relating $W_1(\mu_0,\mu_1)$ to the current equation on metric graphs.
See also \cite{Maas2011GFfinite}, where an alternative metric on the space of probability measures on a finite set $X$ is introduced via modifying the Benamou--Brenier formula in order to have a gradient flow interpretation of the heat flow in a discrete setting.
Different aspects of this metric are studied later in \cite{Erbar2012ricci}, in particular, for the characterization of absolutely continuous curves in the corresponding metric space.

In this section, we study the current equation in a countable and proper metric space for which the induced topology is discrete. More precisely, we show that the current equation can be directly recovered from \cref{thm:lift_BV}, yielding that for a given BV-curve $(\mu_t)$ there exists $v_t \colon X\to \M(X), x \mapsto v_t^{x}$, so that the pair $(\mu_t,v_t)$ satisfies the current equation.
The obtained $v_t$ (or rather $d(x,\cdot) v_t^{x}(\cdot)$) can be interpreted as a velocity field, but unlike the continuous setting, it is a time-dependent positive measure over the space $X$. 
While the (pointwise defined) continuity equation makes sense for BV-curves due to almost everywhere differentiability, the result is more meaningful whenever $(\mu_t)$ is an absolutely continuous curve and therefore completely characterized by the continuity equation.

\vspace{5pt}

\noindent \textbf{Setting.} Throughout this section, $(X,d)$ is a countable and proper metric space whose induced topology is discrete and we adopt the notation $\mu_{t}(x) = \mu_{t}(\{x\})$, $x\in X$. We start by recalling the current equation:

\begin{definition}(current equation)
A family $(\mu_t,v_t)_{t\in I}$ with $\mu_t\in P(X)$, $v_t\colon X\to \M(X)$, is said to satisfy the \emph{current equation} if for every $x\in X$
\begin{align}\label{eq:current}
    \frac{\d}{\d t}\mu_t(x)=\sum_{y \in X} v_t^y(x)\mu_t(y)-\mu_t(x)\sum_{y \in X} v_t^x(y),\quad \text{$\mathcal{L}^1$-a.e. $t\in I$}.
\end{align}
\end{definition}

The following lemma states a useful observation for Wasserstein curves in discrete spaces, whose proof is included in the argument of \cref{thm:ctyeq}.
\begin{lemma} \label{lemma:mu_t_x}Let $(\mu_t)\subset P_1(X)$. If $t\mapsto\mu_t$ is BV or absolutely continuous, then for each $x\in X$, $t\mapsto \mu_t(x)$ is BV or absolutely continuous, respectively. The reverse is also true when all measures $\mu_t$ are supported inside a common bounded set.
\end{lemma}

\begin{theorem}[BV-curves and current equation]\label{thm:ctyeq}
    Let $(\mu_t)\in \mathcal{BV}(I:P_1(X))$ and assume that for each $t\in I=[0,1]$, $\mathrm{supp}(\mu_t)$ is bounded. Then there exists $(v_t)$ so that $(\mu_t,v_t)$ satisfies the current equation.
    If further all measures $\mu_t$ are supported inside a common bounded set, then
    \begin{enumerate}[label=(\roman*), font=\normalfont]
        \item \label{itm:ctyeq_1} For any $(v_t)$ such that the pair $(\mu_t,v_t)$ satisfies the current equation, we have 
        \begin{equation}
            |\dot\mu_t| \leq \sum_{x,y} d(x,y)  v_t^x(y) \mu_t(x), \quad \text{$\mathcal{L}^1$-a.e. $t\in I$}.
        \end{equation}
        \item \label{itm:ctyeq_2} There exists a $(v_t)$ satisfying the current equation such that
        \begin{align}\label{eq:metricspeed}
            |\dot\mu_t|= \sum_{x,y} d(x,y) v_t^x(y) \mu_t(x), \quad \text{$\mathcal{L}^1$-a.e. $t\in I$}.
        \end{align}
    \end{enumerate}
\end{theorem}
\begin{proof}
   \textbf{Proof of the a priori estimate \ref{itm:ctyeq_1}.}
Recall the Kantorovich--Rubinstein theorem, 
    \begin{align}\label{eq:W1_duality}
        W_1(\mu_{t},\mu_{s})
        & = \sup_{\lVert \psi \rVert_{\mathrm{Lip}} \leq 1}  \left| \int \psi \d (\mu_{t} - \mu_{s}) \right|,
    \end{align}
    where the supremum runs over all Lipschitz functions $\psi: X \to \R$ with constant $\lVert \psi \rVert_{\mathrm{Lip}} \leq 1$. 
    Let $(\mu_t,v_t)$ be a pair satisfying the current equation.
    Due to the existence of  $\frac{\d}{\d t}\mu_t(x)$ for $\mathcal{L}^1$-a.e. $t\in I$, we have 
    \begin{align}
        \frac{\mu_t(x) - \mu_s(x)}{t-s}
        & = \frac{\d}{\d t} \mu_t(x) + \varepsilon_x (|t-s|) \\
        & =  \sum_{y} v_t^y(x) \mu_t(y) - \mu_t(x) \sum_{y}  v_t^x(y)  + \varepsilon_x (|t-s|)
    \end{align}
    where the error function $\varepsilon_x (|t-s|)$, which depends on $x$, vanishes as $|t-s| \to 0$.
    Then for any 1-Lipschitz function $\psi$, we obtain 
    \begin{align}
        \int \psi \d (\mu_{t} - \mu_{s})
        & =  \sum_{x} \psi(x)  (\mu_{t}(x) - \mu_{s}(x)) \\
        & = (t-s) \left( \sum_{x,y} \psi(x)   v_t^y(x) \mu_t(y) - \psi(x)  v_t^x(y) \mu_t(x)  + \sum_{x} \varepsilon_x (|t-s|) \right) \\
        & = (t-s) \left( \sum_{x,y} \psi(x)   v_t^y(x) \mu_t(y) - \psi(y)  v_t^y(x) \mu_t(y) + \sum_{x} \varepsilon_x (|t-s|) \right) \\
        & = (t-s)  \left(\sum_{x,y} (\psi(x) - \psi(y))  v_t^y(x) \mu_t(y) +  \sum_{x} \varepsilon_x (|t-s|)\right)\\
        &\leq |t-s| \left( \sum_{x,y} d(x,y)  v_t^y(x) \mu_t(y) +\sum_{x} \varepsilon_x (|t-s|)\right)
    \end{align}
    where in the third step, we simply exchanged the indexes of summation in the second term.
    As all $\mathrm{supp}(\mu_t)$ are confined to a common bounded set, the above summation over $x$ is actually a finite sum. So 
    \begin{equation}\label{eq:ctyeq_proof_bound}
         \left| \int \psi \d (\mu_{t} - \mu_{s})  \right| \leq |t-s|  \sum_{x,y} d(x,y)  v_t^y(x) \mu_t(y) + o (|t-s|).
    \end{equation}
    Since the right-hand side of the equation above no longer depends on the choice of function $\psi$, we can combine \eqref{eq:W1_duality} and \eqref{eq:ctyeq_proof_bound} to get 
    \begin{equation}
        |\dot \mu_t|= \lim_{s \to t} \frac{W_1(\mu_t, \mu_s)}{|t-s|} \leq   \sum_{x,y} d(x,y)  v_t^y(x) \mu_t(y), \quad \text{$\mathcal{L}^1$-a.e. $t\in I$}.
    \end{equation}
    \\
    \textbf{Proof of the existence and \ref{itm:ctyeq_2}}.
    Let $\pi$ be the lift of $(\mu_t)$ given by \cref{thm:lift_BV} with
    \begin{equation}
        |\dot \mu_t|=\lim_{h\searrow 0}\int\frac{d(\gamma_t,\gamma_{t+h})}{h}\d\pi(\gamma).
    \end{equation}
    Denote by $\{\pi^x_t\}$ the disintegration of $\pi$ with respect to $e_t$, i.e., $\pi=\int\pi^x_t \d\mu_t(x)$
    and furthermore by $\{\nu_{t+h}^x\}$ the push-forward $\nu_{t+h}^x\coloneqq (e_{t+h})_\#\pi^x_t$. 
    The goal is to first prove that, for fixed $x$, the measure $$\hat\nu_{t+h}^x\coloneqq\frac{\nu^x_{t+h}\lvert_{X\setminus\{x\}}}{h}$$ converges weakly to some measure $v_t^x$.
    \\ 
    Since $X$ is proper with the discrete topology, the ball $B_1(x)$ contains only finite elements and so
    \[
    r_x\coloneqq \min\{d(x,y):y\neq x\}\wedge 1 > 0.
    \] Thus, 
    \begin{align}
       r_x \hat\nu_{t+h}^x(X)&=r_x\int_{\{\gamma: \gamma_{t+h}\neq x\}}\frac1h\d\pi^x_t\\
        &\leq r_x\int_{\{\gamma: \gamma_{t+h} \in B_1(x) \setminus \{x\} \}}\frac1h\d\pi^x_t +\int_{\{\gamma: \gamma_{t+h}\notin B_1(x)\}}\frac1h\d\pi^x_t\\
        &\leq\int\frac{d(\gamma_t,\gamma_{t+h})}{h}\d\pi^x_t(\gamma)\label{ineq:Sec4_unibdd}.
    \end{align}
   Combining \eqref{ineq:Sec4_unibdd} and the inequality
    \[
    \limsup_{h\searrow0}\int\frac{d(\gamma_t,\gamma_{t+h})}{h}\d\pi^x_t(\gamma)\le \frac{|\dot\mu_t|}{\mu_t(x)}<\infty,
    \]
    we have 
     \begin{align}
        \limsup_{h\searrow 0}\hat\nu_{t+h}^x(X)<\infty.
    \end{align}
    \\
    Next we prove tightness of $\{\hat\nu_{t+h}^x\}_h$. Let $\varepsilon>0$, and for every $M\in \N$ define $\Gamma_M\coloneqq \{\gamma:d(x,\gamma_{t+h})> M\}$. Then
    \begin{align}
         \limsup_{h\searrow0}\frac{\pi^x_t(\Gamma_M)}{h}\le  \limsup_{h\searrow0}\frac1M\int_{\Gamma_M}\frac{d(\gamma_t,\gamma_{t+h})}{h}\d\pi^x_t\le \frac1M \frac{|\dot\mu_t|}{\mu_t(x)}.
    \end{align}
    In particular, $\hat\nu_{t+h}^x(\{y:d(x,y)> M\})\to 0$ uniformly in $h$ when $M\to \infty$.
    Thus, properness of $X$ implies that $\{\hat\nu_{t+h}^x\}_h$ is tight, and for arbitrary weakly convergent subsequence we define $v^x_t$ to be its limit.
    \\
    Note that
    \begin{equation}
       r_x\cdot |\mu_t(x)-\mu_s(x)|\leq W_1(\mu_t,\mu_s). 
    \end{equation}
    So $t\mapsto \mu_t(x)$ is BV or absolutely continuous if $t\mapsto\mu_t$ is BV or absolutely continuous, respectively, which proves \cref{lemma:mu_t_x} as well.
    At $t\in (0,1)$ where $t\mapsto \mu_t(x)$ is differentiable, write
    \begin{align}
        \frac{\d}{\d t} \mu_t(x)&=\lim_{h\searrow 0}\frac{\mu_{t+h}(x)-\mu_t(x)}{h} \\ 
        &=\lim_{h\searrow 0}\frac{(e_{t+h})_{\#}\int\pi^y_t\d\mu_t(y)-\mu_t}{h}(x)\\
        &=\lim_{h\searrow 0}\big[\frac{\int_{\{y\neq x\}}(e_{t+h})_{\#}\pi^y_t\d \mu_t(y)}{h}(x)+\frac{(e_{t+h})_{\#}\pi^x_t(x)\mu_t(x)-\mu_t(x)}{h} \big]\\
       &=\lim_{h\searrow 0} \big[\sum_y \hat{\nu}^y_{t+h}(x)\mu_t(y)-\frac{{\nu}^x_{t+h}(X\setminus\{x\})}{h}\mu_t(x) \big]    \\
      &=\sum_y v_t^y(x)\mu_t(y)-\mu_t(x)\sum_{y} v _t^x (y),
    \end{align}
    where in the last equality we used the assumption that $\mu_t$ is concentrated on finite-many points.
    Finally, for \eqref{eq:metricspeed}, by weak convergence, 
    \begin{align}
        |\dot \mu_t|&=\lim_{h\searrow 0}\int\frac{d(\gamma_t,\gamma_{t+h})}{h}\d\pi(\gamma)=\lim_{h\searrow0}\iint \frac{d(x,\gamma_{t+h})}{h}\d\pi^x_t(\gamma)\d \mu_t(x)\\
        &=\lim_{h\searrow0}\iint \frac{d(x,y)}{h}\d\nu_{t+h}^x(y)\d \mu_t(x)
        \\
        &=\sum_{x}\mu_t(x)\lim_{h\searrow0}\int d(x,y) \d\big(\frac{\nu_{t+h}^x}{h}\big)(y)\\
        &=\sum_{x,y} \mu_t(x) d(x,y)v_t^x(y),
    \end{align}
    where $d(x,\cdot)$ can be regarded as bounded function as all $\mu_t$ are confined to a common bounded set.
\end{proof}

Couple more comments about the continuity equation in the discrete setting are in order. 
First of all, \cref{thm:ctyeq} could be used to prove a Benamou--Brenier type formula for the 1-Wasserstein distance in the discrete setting,
\begin{equation}
    W_1(\mu_0,\mu_1) = \inf_{(v_t,\mu_t)} \int_{0}^1 \sum_{x,y} d(x,y) v_t^x(y) \mu_t(x) \d t,
\end{equation}
where the infimum is taken over all $(\mu_t) \in \mathcal{AC}^1([0,1]:P_1(X))$ with $(v_t,\mu_t)$ satisfying the current equation \eqref{eq:current}. 
In fact, the 1-Wasserstein space over \emph{any} complete and separable metric space is geodesic\footnote{Geodesics are easily obtained by simple interpolation.}, and thus Benamou--Brenier formula follows whenever \cref{thm:ctyeq} is applicable. 
The disadvantage is that Benamou--Brenier formula in such a general form is hardly useful.

If instead, one assumes more structure on the space, for instance, that the space is a discrete metric graph, then one can ask whether the Benamou--Brenier formula holds among all transports that respect the graph structure in a suitable manner. 
As alluded before, such a formulation has been proven to hold by L\'eonard in \cite[Theorem 3.1]{Leonard2016Lazy}. 
We note that the result can be recovered by techniques introduced in this paper in the case of measures with bounded support. 
Indeed, given any $\mu_0$ and $\mu_1$, take $\sigma\in\opt(\mu_0,\mu_1)$. For any $(x,y)\in\mathrm{supp}(\sigma)$, consider a ``discrete geodesic'' $(x=x_1,\dots,x_n=y)$, and perform subsequent linear interpolations between $\delta_{x_i}$ and $\delta_{x_{i+1}}$ to obtain a Wasserstein geodesic $(\mu_t^{xy})$ between measures $\delta_x$ and $\delta_y$. This can be done so that $(\mu_t^{xy})$ has a constant speed.
Now apply \cref{thm:lift_BV} (or simply modify \cref{ex:theBeginning}) to obtain a lift $\pi^{xy}$ of $(\mu_t^{xy})$. Define 
\begin{align}
    \pi \coloneqq \int \pi^{xy}\d\sigma(x,y),
\end{align}
and let $\mu_t\coloneqq (e_t)_\#\pi$ for all $t \in [0,1]$. By construction and \cref{thm:BVSk}, we have that $(\mu_t)$ is a Wasserstein geodesic. 
Finally, it is readily checked that the measures $v_t$ constructed (from this particular $\pi)$ in the proof of \cref{thm:ctyeq} respect the graph structure, that is, $v_t^x(y)=0$ whenever $y$ is not a neighbor of $x$.

One challenge, however, for using our techniques to get more insight into the framework of graphs arises from the fact that it is not clear how to detect those curves on the level of the Wasserstein space which respects the graph structure.
More precisely, it is not clear when a Wasserstein curve has a lift that is concentrated on curves that only jump along the edges of the graph (cf. discussion in \cref{subsec:remarks_results}).
For instance, simply by looking at linear interpolations between measures like in \cref{ex:theBeginning}, one ends up with constant speed Wasserstein geodesics which often don't have any lifts respecting the graph structure.
Notice that for the construction of a pair $(\mu_t,v_t)$ realizing the Wasserstein distance via Benamou--Brenier formula, we do not need to lift arbitrary curves or even geodesics but rather construct a specific Wasserstein geodesic and its lift with the desired endpoints.

\bibliographystyle{amsplain}
\bibliography{ms.bib}

\end{document}